\documentclass[
11pt,  reqno]{amsart}
\usepackage[margin=1.2in,marginparwidth=1.5cm, marginparsep=0.5cm]{geometry}

\usepackage{amsmath,amssymb,amscd,amsthm,amsxtra, esint}

\usepackage{booktabs} 
\usepackage{microtype}
\usepackage{amssymb}
\usepackage{mathrsfs}

\usepackage{upgreek} 

\usepackage{color}
\usepackage[implicit=true]{hyperref}

\usepackage[colorinlistoftodos,prependcaption,color=yellow,textsize=tiny,textwidth=2cm]{todonotes}

\usepackage{xsavebox}

\usepackage{lipsum}

\usepackage{cases}

\allowdisplaybreaks[2]

\definecolor{gr}{rgb}   {0.,   0.69,   0.23 }
\definecolor{bl}{rgb}   {0.,   0.5,   1. }
\definecolor{mg}{rgb}   {0.85,  0.,    0.85}
\definecolor{yl}{rgb}   {0.8,  0.7,   0.}
\definecolor{or}{rgb}  {0.7,0.2,0.2}

\newcommand{\VVert}[1]{{\left\vert\kern-0.25ex \left\vert\kern-0.25ex \left \vert #1 
    \right\vert\kern-0.25ex \right\vert\kern-0.25ex \right\vert}} 

\newcommand{\1}{\hspace{0.2mm}\textup{I}\hspace{0.2mm}}

\newtheorem{theorem}{Theorem} [section]

\newtheorem{lemma}[theorem]{Lemma}
\newtheorem{proposition}[theorem]{Proposition}
\newtheorem{remark}[theorem]{Remark}

\DeclareMathOperator{\Law}{Law}
\DeclareMathOperator{\Id}{Id}

\newcommand{\W}{\mathcal{W}}
\newcommand{\U}{\mathcal{U}}

\newcommand{\dr}{\theta}
\newcommand{\Dr}{\Theta}
\newcommand{\Ha}{\mathbb{H}_a}
\newcommand{\noi}{\noindent}
\newcommand{\Z}{\mathbb{Z}}
\newcommand{\R}{\mathbb{R}}

\newcommand{\T}{\mathbb{T}}
\newcommand{\N}{\mathbb{N}}

\let\P= \undefined
\newcommand{\P}{\mathbf{P}}
\newcommand{\PP}{\mathbb{P}}

\newcommand{\F}{\mathcal{F}}

\newcommand{\nb}{\nabla}

\newcommand{\al}{\alpha}
\newcommand{\be}{\beta}
\newcommand{\dl}{\delta}
\newcommand{\Dl}{\Delta}
\newcommand{\eps}{\varepsilon}

\newcommand{\g}{\gamma}

\newcommand{\ld}{\lambda}
\newcommand{\Ld}{\Lambda}
\newcommand{\s}{\sigma}

\newcommand{\ft}{\widehat}
\newcommand{\wt}{\widetilde}
\newcommand{\cj}{\overline}

\newcommand{\dt}{\partial_t}

\newcommand{\ta}{\theta}

\newcommand{\jb}[1]
{\langle #1 \rangle}

\renewcommand{\l}{\ell}

\renewcommand{\H}{\mathcal{H}}

\newcommand{\les}{\lesssim}
\newcommand{\ges}{\gtrsim}

\newcommand{\ind}{\mathbf 1}

\newcommand{\E}{\mathbb{E}}

\renewcommand{\o}{\omega}
\renewcommand{\O}{\Omega}

\numberwithin{equation}{section}
\numberwithin{theorem}{section}

\newcommand{\too}{\longrightarrow}

\newtheorem*{ackno}{Acknowledgements}

\newcommand{\M}{\mathcal{M}}

\usepackage{tikz}

\usetikzlibrary{shapes.misc}
\usetikzlibrary{shapes.symbols}
\usetikzlibrary{shapes.geometric}
\tikzset{
	dot/.style={circle,fill=black,draw=black,inner sep=1pt,minimum size=0.5mm},
	>=stealth,
	}
\tikzset{
	ddot/.style={circle,fill=white,draw=black,inner sep=2pt,minimum size=0.8mm},
	>=stealth,
	}

\tikzset{decision/.style={ 
        draw,
        diamond,
        aspect=1.5
    }}

\tikzset{dia2/.style
={diamond,fill=white,draw=black,inner sep=0pt,minimum size=1mm},
	>=stealth,
	}

\tikzset{dia/.style
={star,fill=black,draw=black,inner sep=0pt,minimum size=1mm},
	>=stealth,
	}

\colorlet{symbols}{black}
\colorlet{testcolor}{green!60!black}

\def\1{\mathbf{{1}}}

\usetikzlibrary{shapes.misc}
\usetikzlibrary{shapes.symbols}
\usetikzlibrary{snakes}
\usetikzlibrary{decorations}
\usetikzlibrary{decorations.markings}
\definecolor{dblue}{rgb}{0.1, 0.1, 0.9}

\tikzset{
	root/.style={circle,fill=testcolor,inner sep=0pt, minimum size=2mm},		
	dot/.style={circle,fill=black,draw=black, solid,inner sep=0pt,minimum size=0.75mm},
	bdot/.style={circle,fill=blue,draw=dblue, solid,inner sep=0pt,minimum size=0.75mm},
		}
\colorlet{symbols}{blue!90!black}

\makeatletter
\def\DeclareSymbol#1#2#3{\expandafter\gdef\csname MH@symb@#1\endcsname{\tikz[baseline=#2,scale=0.15]{#3}}%
\expandafter\gdef\csname MH@symb@#1s\endcsname{\scalebox{0.6}{\tikz[baseline=#2,scale=0.15]{#3}}}}
\def\<#1>{\csname MH@symb@#1\endcsname}
\makeatother

\DeclareSymbol{sunset}{-3}{\draw (0,0) circle [x radius = 1.5, y radius = 1]; \draw (-1.5,0) node[dot] {} -- (1.5,0) node[dot] {};}
\DeclareSymbol{tadpole}{-3}{\draw (0,0) circle [x radius = 0.75, y radius = 1]; \draw (0,-1) node[dot] {};}

\DeclareSymbol{1}{0}{\draw[white] (-.4,0) -- (.4,0); \draw (0,0)  -- (0,1.2) node[dot] {};}
\DeclareSymbol{2}{0}{\draw (-0.5,1.2) node[dot] {} -- (0,0) -- (0.5,1.2) node[dot] {};}
\DeclareSymbol{3}{0}{\draw (0,0) -- (0,1.2) node[dot] {}; \draw (-.7,1) node[dot] {} -- (0,0) -- (.7,1) node[dot] {};}
\DeclareSymbol{4}{0}{\draw (-0.36,1.2) node[dot] {} -- (0,0) -- (0.36,1.2) node[dot] {}; \draw (-1,1) node[dot] {} -- (0,0) -- (1,1) node[dot] {};}
\DeclareSymbol{30}{-3}{\draw (0,0) -- (0,-1); \draw (0,0) -- (0,1.2) node[dot] {}; \draw (-.7,1) node[dot] {} -- (0,0) -- (.7,1) node[dot] {};}
\DeclareSymbol{31}{-3}{\draw (0,0) -- (0,-1) -- (1,0) node[dot] {}; \draw (0,0) -- (0,1.2) node[dot] {}; \draw (-.7,1) node[dot] {} -- (0,0) -- (.7,1) node[dot] {};}
\DeclareSymbol{32}{-3}{\draw (0,0) -- (0,-1) -- (1,0) node[dot] {}; \draw (0,0) -- (0,-1) -- (-1,0) node[dot] {}; \draw (0,0) -- (0,1.2) node[dot] {}; \draw (-.7,1) node[dot] {} -- (0,0) -- (.7,1) node[dot] {};}
\DeclareSymbol{20}{-3}{\draw (0,0) -- (0,-1);\draw (-.7,1) node[dot] {} -- (0,0) -- (.7,1) node[dot] {};}
\DeclareSymbol{22}{-3}{\draw (0,0.3) -- (0,-1) -- (1,0) node[dot] {}; \draw (0,0.3) -- (0,-1) -- (-1,0) node[dot] {};\draw (-.7,1) node[dot] {} -- (0,0.3) -- (.7,1) node[dot] {};}
\DeclareSymbol{202}{-3}{\draw (-.6625,0) -- (0,-1) -- (.6625,0)  {}; \draw (-1.1,1) node[dot] {} -- (-.6625,0) -- (-0.365,1) node[dot] {}; \draw (0.365,1) node[dot] {} -- (.6625,0) -- (1.1,1) node[dot] {};}

\makeatletter
\def\DeclareSymbol#1#2#3{\expandafter\gdef\csname MH@symb@#1\endcsname{\tikz[baseline=#2,scale=0.15]{#3}}}
\def\<#1>{\csname MH@symb@#1\endcsname}
\makeatother

\DeclareSymbol{X}{-2.4}{\node[dot] {};}
\DeclareSymbol{1}{0}{\draw[white] (-.4,0) -- (.4,0); \draw (0,0)  -- (0,1.2) node[dot] {};}
\DeclareSymbol{2}{0}{\draw (-0.5,1.2) node[dot] {} -- (0,0) -- (0.5,1.2) node[dot] {};}

\DeclareSymbol{sunset}{-3}{\draw (0,0) circle [x radius = 1.5, y radius = 1]; \draw (-1.5,0) node[dot] {} -- (1.5,0) node[dot] {};}
\DeclareSymbol{tadpole}{-3}{\draw (0,0) circle [x radius = 0.75, y radius = 1]; \draw (0,-1) node[dot] {};}

\DeclareSymbol{T0}{-2.7}
 { \draw (0,0) node[ddot]{};}

\DeclareSymbol{T1}{0}
 {  
 \draw (0,0) node[dot]{} -- (0,4) node[ddot] {}; 
 \draw (-3,0) node[dot] {} -- (0,4)node[ddot] {} -- (3,0) node[dot] {};
\draw[line width=1pt] (0, 4)node[ddot, label=above:$r_1$]{} 
.. controls(3,6) and (5,6) ..
(8, 4)node[ddot, label=above:$r_2$]{}
(4, 12) node[label ] {$\underline{j = 1}$};
 }

\DeclareSymbol{T2}{0}
 {\draw (0,0) node[dot]{} -- (0,4) node[ddot] {}; 
 \draw (-3,0) node[dot] {} -- (0,4)node[ddot] {} -- (3,0) node[dot] {};
 \draw (0,-4) node[dot]{} -- (0,0) node[dot] {}; 
 \draw (-3,-4) node[dot] {} -- (0,0)node[dot] {} -- (3,-4) node[dot] {};
\draw[line width=1pt] (0, 4)node[ddot, label=above:$r_1$]{} 
.. controls(3,6) and (5,6) ..
(8, 4)node[ddot, label=above:$r_2$]{}
(4, 12) node[label ] {$\underline{j = 2}$};
 }

\DeclareSymbol{T3}{0}
 {\draw (0,0) node[dot]{} -- (0,4) node[ddot] {}; 
 \draw (-3,0) node[dot] {} -- (0,4)node[ddot] {} -- (3,0) node[dot] {};
 \draw (0,-4) node[dot]{} -- (0,0) node[dot] {}; 
 \draw (-3,-4) node[dot] {} -- (0,0)node[dot] {} -- (3,-4) node[dot] {};
\draw (10,0) node[dot]{} -- (10,4) node[ddot] {}; 
 \draw (7,0) node[dot] {} -- (10,4)node[ddot] {} -- (13,0) node[dot] {};
\draw[line width=1pt] (0, 4)node[ddot, label=above:$r_1$]{} 
.. controls(4,6) and (6,6) ..
(10, 4)node[ddot, label=above:$r_2$]{}
(5, 12) node[label ] {$\underline{j = 3}$};
 }

\tikzstyle{dot1} = [ draw=  gray!00, 
 rectangle, rounded corners, fill=gray!00,  inner sep=1pt, inner ysep=3pt]

\tikzstyle{dot2} = [ draw=  black, 
ellipse, fill=gray!00,  inner sep=1pt, inner ysep=3pt]

\tikzstyle{dot3} = [ draw=  gray!00, 
ellipse, fill=gray!00,  inner sep=1pt, inner ysep=3pt]

\DeclareSymbol{T4}{0}
 {\draw (0,0) node[dot1]{$\phantom{n_2^{(1)} = n^{(2)}}$} -- (0,15) node[dot1] {$\phantom{n^{(1)}}$}; 
 \draw (-13,0) node[dot1] {$n_1^{(1)}$} -- (0,15)node[ddot] {$\phantom{n^{(1)}}$} 
 -- (13,0) node[dot1] {$n_3^{(1)}$};
 \draw (0,-15) node[dot1]{$n_2^{(2)}$} -- (0,0) node[dot1] {$\phantom{n_2^{(1)} = n^{(2)}}$}; 
 \draw (-13,-15) node[dot1] {$n_1^{(2)}$} -- (0,0)node[dot1] {$n_2^{(1)} = n^{(2)}$} -- (13,-15) node[dot1] {$n_3^{(2)}$};
\draw (40,0) node[dot1]{{$n_2^{(3)}$}} -- (40,15) node[dot3]  {$\phantom{n^{(1)} = n^{(3)}}$} ; 
 \draw (27,0) node[dot1] {$n_1^{(3)}$} -- (40,15)node[dot3]  {$\phantom{n^{(1)} = n^{(3)}}$}  -- (53,0) node[dot1] {$n_3^{(3)}$};
\draw[line width=1pt] (0, 15)node[ddot, label=above:$r_1$]{$n^{(1)}$} 
.. controls(10,23) and (27,23) ..
(40, 15)node[dot2, label=above:$r_2$] {$n^{(1)} = n^{(3)}$} ;
 }

\makeatletter
\def\DeclareSymbol#1#2#3{\expandafter\gdef\csname MH@symb@#1\endcsname{\tikz[baseline=#2,scale=0.15]{#3}}}
\def\<#1>{\csname MH@symb@#1\endcsname}
\makeatother

\begin{document}

\baselineskip = 14pt

\title[Asymptotic expansion of the $\Phi^4_2$-measure]
{Low temperature expansion for the Euclidean $\Phi^4_2$-measure}

\author[B.~Gess, K.~Seong, and P.~Tsatsoulis]
{Benjamin Gess, Kihoon Seong, and Pavlos Tsatsoulis}

\address{Benjamin Gess\\
Fakult\"at f\"ur Mathematik, Universit\"at Bielefeld\\
Bielefeld, D-33615\\ 
Max--Planck--Institut f\"ur Mathematik in den Naturwissenschaften\\ 
Leipzig, D-04103}

\email{bgess@math.uni-bielefeld.de}

\address{Kihoon Seong\\
Department of Mathematics\\
Cornell University\\ 
310 Malott Hall\\ 
Ithaca, New York, 14853\\ 
USA }

\email{kihoonseong@cornell.edu}

\address{Pavlos Tsatsoulis\\
Fakult\"at f\"ur Mathematik\\
Universit\"at Bielefeld\\ 
Bielefeld, D-33615}

\email{ptsatsoulis@math.uni-bielefeld.de}



\begin{abstract}
We study asymptotic expansions of the Euclidean $\Phi^4_2$-measure in the low-temperature regime. In particular, this extends the asymptotic expansions of Gaussian function space integrals developed in Schilder (1966) and Ellis and Rosen (1982) to the singular setting, where the field is no longer a function, but just a distribution. As a consequence, we deduce limit theorems, specifically the law of large numbers and the central limit theorem for the $\Phi^4_2$-measure in the low-temperature limit.
\end{abstract}

\maketitle

\tableofcontents

\setlength{\parindent}{0mm}
\setlength{\parskip}{6pt}

\section{Introduction}

\subsection{Asymptotic expansions of the $\Phi^4_2$-measure}

The fundamental objects of constructive quantum field theory in the two-dimensional case set on the torus 
$\T^2 = (\R/2\pi\Z)^2$ are so-called $\Phi^4_2$ probability measures of the form\footnote{For simplicity we disregard the need for renormalization for the moment.}
\begin{align}
d\rho_\eps(\phi)&=Z_\eps^{-1} \exp\Big\{ -\frac{1}{\eps} H(\phi) \Big\}\prod_{x \in \T^2}d\phi(x),
\label{Gibbs}
\end{align}


\noi
where  $Z_\eps$ is the partition function, $\eps$ is a positive (temperature) parameter, and $\prod_{x\in\T^2}d\phi(x)$ is the informal Lebesgue measure on fields $\phi: \T^2 \to \R$. Here, the Hamiltonian $H(\phi)$ associated with the $\Phi^4_2$-measure \eqref{Gibbs} is given by\footnote{By convention, 
we endow $\T^2$ with the normalized Lebesgue measure $dx_{\T^2}= (2\pi)^{-2} dx$.}:
\begin{align}
H(\phi)&=\frac 12\int_{\T^2} |\nb \phi|^2 dx+\frac{1}{4} \int_{\T^2} (|\phi|^2-1)^2 dx \notag \\
&=\frac{1}{2} \int_{\T^2} |\nb \phi|^2 dx+ \textbf{V}(\phi),
\label{Ham0}
\end{align}

\noi
where $\textbf{V}(\phi):=\int_{\T^2} V(\phi) dx$ and $V(\phi):=\frac{\ld}{4}(|\phi|^2-1)^2$ is the double-well potential. The parameter $\ld>0$ is the so-called coupling constant that measures the strength of the interaction potential. The construction of such measures was a crucial achievement of the constructive field theory program during the 70’s and 80’s \cite{Nelson, GlimmJ68, GlimmJ73, GJS76, GJS762, FO76, GlimmJ87}.

One is then interested in estimating observables $F$ under this field, that is,
\begin{align}
\int F(\phi) \rho_\eps (d\phi)=Z_\eps^{-1}\int F(\phi) \exp\Big\{-\frac{1}{\eps} \textbf{V}(\phi) \Big\} \mu_\eps(d\phi),
\label{ob0}
\end{align}

\noi
where $F$ is a functional with sufficiently many Fr\'echet  derivatives and $\mu_\eps$ is the (log-correlated) Gaussian free field with covariance operator $\eps(-\Dl)^{-1}$ on $\T^2$, formally given by\footnote{Here $Z_\eps$ denotes different normalizing constants that may differ from one line to line.} 
\begin{align}
\mu_\eps(d\phi) 
= Z_{\eps}^{-1} \exp\Big\{-\frac 1{2\eps}  \jb{-\Dl\phi,\phi}_{L^2(\T^2)}    \Big\} \prod_{x\in \T^2} d\phi(x).
\end{align}
For example motivated from applications to synchronization by noise, see Subsection \ref{SUBSUB:SYN} below for more details, the observables $F$ in \eqref{ob0} have to be allowed to be polynomially growing.

\noi
The main result of this paper is to obtain asymptotic expansions  of \eqref{ob0} in the low-temperature limit $\eps \to 0$, of the form,   
\begin{align}
\int F(\phi)\exp\Big\{-\frac 1\eps \textbf{V}(\phi)  \Big\} \mu_\eps(d\phi)
= \sum_{j=0}^ka_j \eps^{\frac{j}{2}}+O(\eps^{\frac{k+1}{2}}),
\label{exp00}
\end{align}

\noi
with explicit constants $\{a_j\}_{j=0}^k \subset \R$, see  Theorem \ref{THM:1} below.

For instance, the coefficient $a_0$ is given by $\sum_{w \in \M}\dr_{\text{re}}(w)F(w)$, where $\{w\}_{w\in \M}$ is the set of minimizers of the Hamiltonian $H$, and $\dr_{\text{re}}(w)$ represents the asymptotic mass for the $\Phi^4_2$-measure concentrating on $w$ as $\eps \to 0$, which is explicitly given by the Carleman-Fredholm determinant. Hence, the case $k=0$ corresponds to the law of large numbers for the $\Phi^4_2$-measure, see also Theorem \ref{THM:2} below.

\noi
We further note that, for $k=1$, \eqref{exp00} implies a central limit theorem in the sense that 
 $$\lim_{\eps \to 0}\Bigg[\varepsilon^{-\frac{1}{2}}\bigg(\int F(\phi)\exp\Big\{-\frac 1\eps \textbf{V}(\phi)  \Big\} \mu_\eps(d\phi) - a_0 \bigg) \Bigg]=a_1= \sum_{w\in \M} \dr_{\text{re}}(w)  \int \jb{DF(w),v }  \mu_w(dv)$$

\noi
for a family of explicit Gaussian measures $ \mu_w$. 
Further discussion and another central limit theorem is given in Theorem \ref{THM:3} below.


We emphasize that the expansion \eqref{exp00} is of completely different nature than asymptotic expansions in terms of the coupling constant $\lambda>0$ in perturbative quantum field theory, see also Subsection \ref{SUBSUB:per} below.

In contrast to the one-dimensional case, in two spatial dimensions, the Gaussian measure $\mu_\varepsilon$ is not concentrated on a space of functions, but just on distributions. Due to this singularity, the proof presented in this work relies on a novel combination of the arguments developed for the non-singular case in Ellis-Rosen \cite{ER1} with techniques from singular SPDEs, and renormalization \cite{BG,Hairer}. In particular, this entails the derivation of a quantified Varadhan lemma by a direct use of the variational method \cite{BG}, the renormalization of Fredholm determinants appearing in the change of variables in Gaussian measures, the use of the variational method to prove the convergence of the expansion in the ultraviolet limit, and the exponential concentration of mass on the enhanced model space. For a more detailed account of the proof see Subsection  \ref{sec:structure_proof} below.





\subsection{Main results}

In this subsection, we introduce the main results. Prior to that, we provide a review of the construction of the $\Phi^4_2$-measure and relevant notations. We start by introducing the renormalization of the $\Phi^4_2$-measure. This is standard material and details can be found, for example, in \cite{PS, OTh}.

The (massive) Gaussian free field $\mu$ is the Gaussian measure on $\mathcal{D}'(\T^2)$ with covariance
\begin{align}
&\E_{\mu}\big[ \jb{f,\phi} \jb{g,\phi} \big] =\jb{f, (1-\Dl)^{-1}g}_{L^2(\T^2)},
\end{align}

\noi
which is formally given by
\begin{align*}
d\mu(\phi)=Z^{-1}\exp\Big\{-\frac 1{2} \big\langle(1-\Dl)\phi,\phi\big\rangle_{L^2(\T^2)}  \Big\} \prod_{x\in \T^2} d\phi(x).
\end{align*} 

\noi
The Gaussian free field $\mu$ on $\T^2$ can be understood as follows. Let $u(x;\o)$ be the following Gaussian Fourier series
\begin{align}
\o\in \O \longmapsto u(x;\o)=\sum_{n \in \Z^2} \frac{g_{n}(\o) }{\jb{n}}e^{ i n\cdot x} \in \mathcal{D}'(\T^2),
\label{RFOU}
\end{align}

\noi
where $\jb{\cdot}=(1+|\cdot|^2)^{\frac 12}$ and $\{ g_n \}_{n \in \Z^2}$ is a sequence of mutually independent standard complex-valued Gaussian random variables on a probability space $(\O,\F,\PP)$ 
conditioned on  $g_{-n} = \cj{g_n}$ for all $n \in \Z^2$. By denoting the law of a random variable $X$ by $\Law(X)$  (with respect to the underlying probability measure $\PP$), we have
\begin{align*}
\Law_{\PP} (u) = \mu
\end{align*}

\noi
for $u$ in \eqref{RFOU}. Note that  $\Law (u) = \mu$ is supported on $H^{s}(\T^2)$
for $s < 0$ but not for $s \geq 0$ and more generally in $B^s_{p,q}(\T^2)$ for any $1 \le p, q \le \infty$ and $s < 0$. Given $N\in \N$, we define the frequency projector $\P_N $ onto the frequencies $\{ |n|\le N\}$ as follows
\begin{align}
\P_N f=\sum_{|n|\le N}\ft f(n) e^{in\cdot x}.
\label{defpr}
\end{align}

\noi
We set $f_N:=\P_N f$. In this paper, we face the following ultraviolet (small-scale) problems. For any $\phi$ under the free field $\mu$ and each fixed $x \in \T^2$, $\phi_{N}(x)$ is a mean-zero Gaussian random fields with 
\begin{align}
c_N=\E_{\mu}\Big[|\phi_{N}(x)|^2\Big] &= \sum _{\substack{|n| \le N }} \frac{1}{\jb{n}^2}=\<tadpole>_{N}
\sim \log N \too \infty, \label{tadpole1}
\end{align}

\noi
as $N \to \infty$. Note that $\<tadpole>_{N}$ is  independent of $x\in \T^2$ due to the stationarity of the Gaussian free field $\mu$. We define the Wick powers $:\! \phi_{N}^2 \!:, :\! \phi_{N}^3 \!: $, and $:\! \phi_{N}^4 \!: $ as follows 
\begin{align}
:\! \phi_{N}^2 \!: \, &= \phi_{N}^2 - \<tadpole>_{N} \label{Wickqu}\\
:\! \phi_{N}^3 \!: \, &= \phi^3_{N} -\<tadpole>_{N} \phi_{N} \label{Wickcu}\\
:\! \phi_{N}^4 \!: \, &= \phi^4_{N} -6\<tadpole>_{N} \phi_{N}^2+3\<tadpole>_N^2 \label{Wickquar}
\end{align}

\noi
which implies that every divergent contribution is removed from the expressions so that the limit of \eqref{Wickqu} ,\eqref{Wickcu}, and \eqref{Wickquar} under $\mu$ exist. Therefore, we study the subsequent renormalized interaction potential
\begin{align}
\mathcal{V}_N(\phi)  = \mathcal{V}_{1,N}(\phi)  + \mathcal{V}_{2,N}(\phi)
\label{POTRE}
\end{align}

\noi
where the contracted graphs causing divergent contributions are excluded during the renormalization process, that is,
\begin{align}
\mathcal{V}_{1,N}(\phi):&=\frac 14 \int_{\T^2} :\! \phi_N^4  \!: dx  -\frac 12 \int_{\T^2} :\! \phi_N^2  \!:  dx +\frac 14 \notag \\
\mathcal{V}_{2,N}(\phi):&=-\frac 1{2} \int_{\T^2} :\! \phi_N^2 \!: dx. 
\label{renpote}
\end{align}

\noi
Then, for any given $\eps>0$, we define the truncated and renormalized $\Phi^4_{2}$-measure by
\begin{align}
d\rho_{\eps,N}(\phi)&=Z_{N,\eps}^{-1} \exp\Big\{ -\frac{1}{\eps}\mathcal{V}_N(\phi)   \Big\} \mu_\eps(d\phi),
\label{renormea}
\end{align}

\noi
where $\mu_\eps$ is the Gaussian free field with covariance operator $\eps(1-\Dl)^{-1}$ and  the partition function $Z_{N, \eps}$ is given by
\begin{align}
Z_{N,\eps}=\int \exp\Big\{-\frac{1}{\eps}\mathcal{V}_N(\phi)   \Big\} \mu_\eps(d\phi).
\label{Part1}
\end{align}

\noi
Then, $\rho_{\eps, N}$ is a finite-dimensional approximation of the $\Phi^4_2$-measure. 

\begin{remark}\rm
For technical considerations, we employ a massive Gaussian free field as our reference measure. This involves introducing mass into the covariance $(1-\Delta)^{-1}$ to address the problem of zeroth Fourier mode degeneracy. To maintain consistency with \eqref{Ham0}, we subtract $\frac{1}{2}\int_{\mathbb{T}^2} \phi^2 dx$ from $\mathcal{V}(\phi)$ as in \eqref{POTRE}.
\end{remark}

\begin{proposition}[Nelson \cite{Nelson}]\label{PROP:NEL}
For every $\eps>0$ and $1 \le p<\infty$, the renormalized potential $\mathcal{V}_{N} $  in \eqref{POTRE} converges to some limit $\mathcal{V}$ in $L^p(d\mu_\eps)$ as $N\to \infty$. Moreover, the truncated $\Phi^4_{2}$-measure $\rho_{\eps,N}$ in \eqref{renormea} converges weakly to a non-Gaussian measure $\rho_\eps$ on $\mathcal{D}'(\T^2)$  as $N\to \infty$, given by
\begin{align}
d\rho_{\eps}(\phi)&=Z_{\eps}^{-1} \exp\Big\{ -\frac{1}{\eps}\mathcal{V}(\phi)   \Big\} \mu_\eps(d\phi)
\label{Phi2nN}
\end{align}
	
\noi
where $Z_{N,\eps} \to Z_{\eps}$ as $N \to \infty$ and 
\begin{align}
Z_{\eps}= \int \exp\Big\{-\frac{1}{\eps}\mathcal{V}(\phi)   \Big\} \mu_\eps(d\phi).
\label{PART12}
\end{align}

\noi
Furthermore, the limiting $\Phi^4_{2}$-measure $\rho_\eps$ is equivalent to the log-correlated Gaussian  field~$\mu_\eps$.

\end{proposition}

We are now ready to present the main results of this work.


\begin{theorem}[Asymptotic expansion]\label{THM:1}
Let $k \ge 1$, $\eta>0$, and let $F$ be a function that is $k+1$-times continuously differentiable on $C^{-\eta}(\T^2)$, with derivatives exhibiting at most polynomial growth rate. Then, the $\Phi^4_{2}$-measure allows an asymptotic expansion in terms of the temperature parameter $\eps$, that is,  there exists $\eps_0>0 $ and a sequence $\{a_j\}_{j=1}^k$ depending on $F$ such that 
\begin{align}
\int F(\phi)\exp\Big\{-\frac 1\eps \mathcal{V}(\phi)  \Big\} \mu_\eps(d\phi)
= \sum_{j=0}^ka_j \eps^{\frac{j}{2}}+O(\eps^{\frac{k+1}{2}})
\label{asymexp}
\end{align}

\noi
for every $0<\eps\le \eps_0$.


The coefficients $a_j$, are explicitly given by
\begin{align*}
a_j=\sum_{w\in \M}   \dr_{\textup{re}}(w)\frac{1}{j!}\int Q_j(w,v)  \mu_{w}(dv),
\end{align*}
\noi
see \eqref{coeff}, where $\{\mu_w\}_{w \in \M}$ is a family of new Gaussian measures with covariances  $\big(\nb^{(2)} H(w)\big)^{-1}=(-\Dl+\nb^{(2)}V(w))^{-1}$, $\dr_{\text{re}}(w)$ is the Carleman–Fredholm determinant,  and $Q_j(w,v)$ is the $j$-th Fr\'echet derivative $D^{(j)}G(w,0)(v,\dots,v)$ of a new density $G(w,v)$ of the $\Phi^4_2$-measure with respect to the Gaussian measure $\mu_w$, 
\begin{align*}
G(w, \sqrt{\eps}v)&=F(w+\sqrt{\eps}v ) \exp\bigg\{-\frac \eps4 \int_{\T^2}  :\! v^4 \!:_w  dx+\frac{3\eps c}{2} \int_{\T^2} :\! v^2 \!:_w dx \\
&\hphantom{XXXXXXXXXX}-\sqrt{\eps}\int_{\T^2}  :\! v^3 \!:_w  \cdot w dx+ 3\sqrt{\eps}c \int_{\T^2} v  wdx   \bigg\},
\end{align*}

\noi
for some constant $c\in \R$, see \eqref{W11}. Here, the Wick powers are taken with respect to the Gaussian measures $\mu_w$.

\end{theorem}

The inclusion of observables $F$ with polynomial growth rates in \eqref{asymexp} is motivated, for example, from applications to synchronization by noise for the scalar-valued dynamical $\Phi^4_2$-model \eqref{NLH}, see Subsection \ref{SUBSUB:SYN} below.


Following the proof of Theorem \ref{THM:1}, it is possible to also derive the corresponding laws of large numbers (Theorem \ref{THM:2}) and central limit theorems (Theorem \ref{THM:3}). We point out that, while the asymptotic expansion \eqref{expa001} in Theorem \ref{THM:1} requires the assumption that $F$ is $(k+1)$ times Fréchet differentiable, the subsequent theorems only require that $F$ is a continuous functional.

\begin{theorem}[Law of large numbers]\label{THM:2}
Let $F$ be a continuous functional on $C^{-\eta}(\T^2)$, with having at most polynomial growth rate. Then, there exists $\{b(w)\}_{w\in \M}$ such that 
\begin{align*}
\lim_{\eps \to 0}Z^{-1}_\eps\int F(\phi)\exp\Big\{ -\frac{1}{\eps} \mathcal{V}(\phi)   \Big\}  \mu_\eps(d\phi)=\sum_{w\in \M} b(w) F(w)
\end{align*}

\noi
as $\eps \to 0$, where
\begin{align}
b(w):=\frac{ \dr_{\textup{re}}(w)}{\sum_{\bar w\in \M}  \dr_{\textup{re}}(\bar w)}
\label{BW0a}
\end{align}

\noi
and $\dr_{\textup{re}}(w)$ is the Carleman–Fredholm (renormalized) determinant from Lemma \ref{LEM:redet} and \eqref{REDET0}.

\end{theorem}

Theorem \ref{THM:2} shows that $\Phi^4_{2}$-measure converges in distribution to the probability measure $\sum_{w\in \M}   b(w) \dl_w$ concentrated on the minima $\M$ with masses $b(w) \in [0,1]$. Namely, if $X_n$ is a random variable whose law is given by $\Phi^4_{2}$-measure with temperature $\frac{1}{n}$, then $X_n$ converges weakly to $X$ whose law is $\sum_{w\in \M}b(w) \dl_w $.

We next study the fluctuations of $\Phi^4_{2}$-measure  around its limiting distribution $\sum_{w\in \M}b(w) \dl_w $. The following theorem says that these fluctuations are a mixture of the new Gaussian measures $\mu_w$ for each $w\in \M$. It is worth noting that for a sufficiently small $\delta > 0$, we can associate a unique closest point $\pi(\phi)$ in $\M$ to any $\phi$ satisfying $\text{dist}(\phi, \M) < \delta$. In cases where $\text{dist}(\phi,\M) \ge \delta$, we set $\pi(\phi) = 0$. Therefore, we can define the projection map $\pi$ onto $\M$.  Given that, according to Theorem \ref{THM:2}, the samples from the $\Phi^4_{2}$-measure converge to $\M$ as $\epsilon\to0$ and so it implies that $(\phi-\pi(\phi))$ tends to be zero under the $\Phi^4_2$-measure as $\eps \to 0$.

\begin{theorem}[Central limit theorem]\label{THM:3}
Let $F$ be a continuous functional on $C^{-\eta}(\T^2)$, with having at most polynomial growth rate. Then, there exists $\{b(w)\}_{w\in \M}$ such that 
\begin{align*}
&\lim_{\eps \to 0}Z^{-1}_\eps\int F\big(\sqrt{\eps}^{-1}(\phi-\pi(\phi) )   \big)\exp\Big\{ -\frac{1}{\eps} \mathcal{V}(\phi)   \Big\}  \mu_\eps(d\phi)\\
&=\sum_{w\in \M}   \int   F(v) b(w)  \mu_w(dv)=\int F(v) \nu(dv)
\end{align*}

\noi
as $\eps \to 0$, where $b(w)$ is as in \eqref{BW0a} and
\begin{align*}
\nu:=\sum_{w\in \M} b(w) \mu_w.
\end{align*}

\end{theorem}

This shows that the fluctuation around the limiting distribution is the $\sum_{w}b(w)$-mixture of Gaussian measures $\mu_w$ in each $w\in \M$. In particular,
Theorem \ref{THM:3} exhibits that if $X_n$ is a random variable whose law is given by $\Phi^4_{2}$-measure with temperature $\frac{1}{n}$, then the law of $\sqrt{n}\big(X_n-\pi(X_n)\big)$ converges weakly to $\sum_{w\in \M} b(w) \mu_w$.



\subsection{Structure of the proof}\label{sec:structure_proof}

In order to appreciate the challenges arising in the singular setting, let us first recall the strategy of proof of \eqref{exp00} in the non-singular case. Formally speaking, the Laplace principle implies that 
the $\Phi^4_2$-measure concentrates on the minimizers $\{w\}_{w\in \M}$ of the Hamiltonian $H$ with masses $\dr(w)$. 
In first order, one thus obtains the law of large numbers, corresponding to $a_0 = \sum_{w \in \mathcal{M}} \dr(w) F(w)$ in \eqref{exp00}. Going to higher order requires a finer analysis of the fluctuations. Let $w$ be a global minimum point for $H$ in \eqref{Ham0}. This implies $\nb H(w)=0$ and $\nb^2 H(w) \ge 0$. Hence, for $v$ small,
\begin{align}
 H(w+v)=H(w)+\jb{\nb H(w),v}+\frac{1}{2}\jb{\nb^2 H(w)v,v}+\mathcal{E}(w;v),
\label{2ndor}
\end{align}
\noi
where  $H(w)=\nb H(w)=0$ and $\mathcal{E}(w;v)$ is a higher order error term. If $\nb^2 H(w)>0$, then using \eqref{2ndor} in \eqref{ob0} implies that one can expand around a new Gaussian measure $\mu_w$, determined by the covariance operator $\big(\nb^{(2)} H(w) \big)^{-1}=(-\Dl+\nb^{(2)}V(w) )^{-1}$ as follows: if $\phi$ is close to $\M$, namely, $\phi=w+v$, where $w\in \M$ and $v$ has a small size,  then we have, informally,
\begin{align}
& \int F(\phi)\exp\Big\{-\frac 1\eps H(\phi)  \Big\} \prod_{x } d\phi(x) \notag \\
&\approx  \int \int F(w+v)\exp\Big\{ -\frac{1}{\eps}\mathcal{E}(w;v)  \Big\} \exp\Big\{-\frac{1}{2\eps} \jb{\nb^{(2)}H(w)v,v  }     \Big\} \prod_{x} dv(x) \s(dw) \notag \\
&= \int  \int F(w+\sqrt{\eps}v)\exp\Big\{ -\frac{1}{\eps}\mathcal{E}(w;\sqrt{\eps} v)  \Big\}\det(\nb^{(2)}H(w))^{-\frac 12}\mu_\omega(dv) \s(dw), 
\label{URI2}  
\end{align}
\noi
where $\s(dw)$ is the surface measure on $\M$ and we used the change of variable $v\mapsto\sqrt{\eps }v$ in the last line. 
In the non-singular case, this applies for example to the $\Phi^4_1$-measure, and the details of this argument were carried out by Ellis and Rosen in \cite{ER1}. In this case, the sample from the Gaussian free field on $\T$ is a continuous function and no renormalization is required.  

In contrast, in the two-dimensional case $\T^2$ the free field $\mu_\eps$ is not supported on a space of functions anymore, causing the need for renormalization, that is, the subtraction of diverging constants. The first step in the proof of the asymptotic expansion consists in proving the exponential concentration of mass onto the minima of the potential $H$. Since the derivation of this fact from the large deviations principle for the $\Phi^4_2$-measure would be restricted to bounded observables $F$, and would lead to non-quantitative estimates with respect to the parameters of the model, we develop a different and novel argument consisting of a direct use of the variational approach by Barashkov and Gubinelli \cite{BG}, to obtain quantitative estimates on the mass away from the minimizers, see Proposition \ref{PROP:Var} below. 
We recall that the inclusion of polynomially growing observables is important, for example, in the context of synchronization by noise, see Subsection \ref{SUBSUB:SYN} below. 
This direct argument further allows the derivation of a quantified version of Varadhan's lemma including the explicit control of the estimates on parameters of the model like the viscosity and the size of the periodic domain.

Given the concentration of mass on sufficiently small neighborhoods of the minimizers, the proof proceeds by rewriting \eqref{ob0} in form of a perturbation around the manifold of minimizers. See Lemma \ref{LEM:expN}. This relies on a change of variables and the introduction of corresponding new Gaussian measures $\mu_w$, determined by the covariance operator $\big(\nb^{(2)} H(w) \big)^{-1}$, as indicated in \eqref{URI2}. In higher dimensions ($d\ge 2$), this, however, leads to several difficulties: The first key difficulty arises from the divergence of the Fredholm determinant linked to the partition function of the new Gaussian measures $\mu_w$, that is,
\begin{align*}
\det\Big(\Id+(-\Dl)^{-1}\nb^{(2)}V(w)\Big)=\infty,
\end{align*}
\noi
since, unlike in $L^2(\T)$, the operator $\big(\Id+(-\Dl)^{-1}\nb^{(2)}V(w)\big)$ is not of trace class in $L^2(\T^2)$. This difficulty is overcome in the present work by passing to the renormalization of the Fredholm determinant in form of the Carleman–Fredholm determinant
\begin{align*}
\det\big(\Id + (-\Dl)^{-1}\nb^{(2)}V(w) \big)e^{- \text{Tr} \big[ (-\Dl)^{-1}\nb^{(2)}V(w) \big]}.
\end{align*}
\noi
See Lemma \ref{LEM:expN} and \ref{LEM:redet}. Note that in contrast to the Fredholm determinant, the Carleman–Fredholm determinant is well-defined for Hilbert–Schmidt perturbations of the identity, thus going beyond the smaller class of trace-class perturbations. Recall that $(-\Dl+1)^{-1}$ is Hilbert–Schmidt both in two and three dimensions.


The next issue is to show the ultraviolet stability for the new Gibbs type measures
\begin{align}
 \int F(w+\sqrt{\eps}v)\exp\Big\{ -\frac{1}{\eps}\mathcal{E}(w;\sqrt{\eps} v)  \Big\}\det(\nb^{(2)}H(w))^{-\frac 12}\mu_\omega(dv)  \label{NGIB}
\end{align} 
\noi 
arising in \eqref{URI2}. Since the Gaussian measures $\mu_w$ are not supported on a space of functions, we need a small-scale (ultraviolet) cutoff $\textbf{P}_N$ to regularize the field, see  \eqref{defpr} below. This leads to the $N$-dependent expansion 
\begin{align}
\eqref{URI2}=\sum_{j=0}^ka_{j}^N \eps^{\frac{j}{2}}+O_N(\eps^{\frac{k+1}{2}}).
\label{URI1}
\end{align}
\noi
The challenge then is to prove the convergence of each $a_j^N$ as $N\to \infty$, and to estimate the error $O_N(\eps^{\frac{k+1}{2}})$ uniformly in $N$. The first of these is resolved in the present work by invoking again the variational method for each $w\in \mathcal{M}$. 

The remaining difficulty is to estimate $O_N(\eps^{\frac{k+1}{2}})$ uniformly in $N$. The proof of this relies on establishing the concentration of the new Gaussian measures $\mu_w$ with covariance $(-\Dl+\nb^{(2)}V(w))^{-1}$ on model space, see Lemma \ref{LEM:out1}. Specifically, we focus on the enhanced data set comprised of Wick powers, denoted as $\Xi_{w}=\big(\<1>_{w}, \<2>_{w}, \<3>_{w}, \<4>_{w} \big)$, which is associated with the new Gaussian measure $\mu_w$. We prove the  concentration of Gaussian functional integrals on the model space related to $\Xi_w$, by combining both the variational argument (Lemma \ref{LEM:UNIF}) and Fernique’s theorem on the enhanced model space (Lemma \ref{LEM:Fern}). 

The main contribution of this work is to address these issues and, thereby, to extend the asymptotic expansion by Ellis and Rosen \cite{ER1} to the singular setting. 

\subsection{Motivation and comments on the literature}

\subsubsection{Construction of $\Phi^4$ measures}

The construction of $\Phi^4$ measures was a major achievement of constructive quantum field theory in the 20th century. See for example \cite{Nelson, GlimmJ68, GlimmJ73,
GJS76, GJS762, FO76, Imbrie1, Imbrie2, GlimmJ87}. As mentioned above, the measure in two spatial dimensions was constructed by Nelson in \cite{Nelson}. In three spatial dimensions, the initial breakthrough in \cite{GlimmJ73} was achieved by Glimm and Jaffe, with subsequent results emerging in later works such as \cite{BDH, MW17, GHb, BG, CWG, BDlog}. In higher dimensions, there are triviality results showing that the Gibbs measures become Gaussian. For $d\ge 5$, such results were established by Aizenman and Fröhlich \cite{Aiz, Fro82}, while $d=4$ was more recently solved by Aizenman and Duminil-Copin in \cite{ADC}.

\subsubsection{Perturbative quantum field theory}\label{SUBSUB:per}

Asymptotic expansions of the $\Phi^4$-measure with respect to the coupling constant $\ld>0$ play a crucial role in interpreting perturbative quantum field theory, where observables are computed as formal power series in the regime of small coupling constant $\ld>0$ by 
\begin{align}
\int F(\phi) \exp\Big\{-\frac{\ld}{4} \int \big(|\phi|^{2}-1\big)^{2} dx -\frac{1}{2}\int  |\nb \phi |^2 dx  \Big\} \prod_{x} d\phi(x)=\sum_{j=0}^\infty b_j \ld^j.
\label{expal}
\end{align}
\noi
It is well known that this formal expansion is a divergent series for any fixed $\ld>0$, which was proved by Jaffe \cite{Jaffe65}. Even though the expansion is divergent, this perturbative approach is widely employed in other quantum field theories in physics due to its consistency with experiments with extremely high accuracy. This is so-called asymptoticity means that a satisfactory approximation can be achieved by truncating the series after a fixed finite number of terms as $\ld$ approaches zero. Hence, the goal of perturbative theory is to prove that quantum field theories allow for a perturbative expansion, where the coefficients appearing in the expansion \eqref{expal} converge to finite limits as the ultraviolet cutoff is removed. Regarding the $\Phi^4_2$ theory, Dimock \cite{Dimock} showed that the perturbative expansion is asymptotic. In a recent work, Shen-Zhu-Zhu \cite{SZZ0} presented an alternative method relying on singular stochastic PDEs.
The asymptotic nature of the $\Phi^4_3$ theory was also established by Feldman-Osterwalder \cite{FO76}, and an alternative recent proof has been given in Berglund-Klose \cite{BK22}.

We emphasize that the two asymptotic expansions presented in \eqref{exp00} and \eqref{expal} are of completely different nature. Indeed, by rescaling we note that
\begin{align}
&\int F(\phi)\exp\Big\{-\frac 1\eps \mathcal{V}(\phi)  \Big\} \mu_\eps(d\phi) \notag \\
&= \int F(\sqrt{\eps}\phi)\exp\bigg\{-\frac{\eps}{4}\int\big(|\phi|^{2}-\eps^{-1}\big)^{2}dx-\frac{1}{2}\int|\nb\phi|^{2}dx\bigg\}\prod_{x}d\phi(x),
\end{align}
which demonstrates the different nature of the scaling than \eqref{expal}. In particular, the corresponding coefficients of the expansion are different, for example, the leading-order term $b_0$ (when $\lambda=0$) in \eqref{expal} is associated with the Gaussian free field. In contrast, in \eqref{exp00} the leading-order term corresponds to the concentration of the $\Phi^4_2$-measure around the minimizers of Hamiltonian. 


%

\subsubsection{Stochastic quantization}\label{SUBSUB:Quan}
Sampling from $\Phi^4$ measures is crucial for computing observables. Thanks to the ergodicity $\Law(u_\eps(t)) \to \Phi^4$ for $t\to \infty$, where $u_\eps$ is the solution to the dynamical $\Phi^4$ model
\begin{align}
\dt u_\eps -  \Dl  u_\eps + u_\eps^{3}-\infty \cdot u_\eps  =  \sqrt {2\eps}\xi, \qquad (x,t) \in \T^d\times \R_{+},
\label{NLH}
\end{align} samples from the $\Phi^4$-measures can be obtained by sampling from the stochastic PDE. Here, $\eps$ is a positive (temperature) parameter as in \eqref{Gibbs} and $\xi=\xi(x,t)$ denotes the (Gaussian) space-time white noise on $\T^d \times \R_+$. In fact, this was one of the motivations to introduce stochastic quantization of Euclidean quantum field theories by Parisi and Wu in \cite{PW}. 

\noi
The rigorous analysis of \eqref{NLH} and its renormalization has been a notable achievement in recent years \cite{DPD, Hairer, GIP, CC, RZZ, MW17, TsW, GH0, SZZ0}. We further note that the parabolic $\Phi^4$-model \eqref{NLH} also emerges in appropriate continuum limits of Ising-type models \cite{MW17b, GMW23}.

\subsubsection{Large deviations for singular SPDEs}

In the small noise regime $\eps \to 0$, the large deviation principle (LDP) for the family of solutions $\{u_\eps\}_{\eps >0}$ for the dynamical $\Phi^4$ model \eqref{NLH} has been extensively studied. In two and three dimensions, the LDP was proved by Hairer and Weber \cite{HaWeb}, relying on the theory of regularity structures. A feature of the result in \cite{HaWeb} is that the diverging counter terms in \eqref{NLH} are shown to disappear on the level of the large deviation rate function. 

The large deviation principle for the $\Phi^4_2$-measure in infinite volume has been studied by Barashkov-Gubinelli \cite{BG1}, and the LDP for the $\Phi^4_3$-measure on $\T^3$ has been obtained by Barashkov \cite{Bar}, based on the Boué–Dupuis formula as in \cite{BG}, and in Klose-Mayorcas \cite{KM04} via stochastic quantization, respectively. Similar to the dynamical $\Phi^4$ model \eqref{NLH}, the diverging counter terms are shown not to appear in the corresponding large deviations rate function. 

\subsubsection{Synchronization by noise}\label{SUBSUB:SYN} In \cite{GT} the first and third authors established synchronization by noise for the vector-valued dynamical $\Phi^4_1$-model \eqref{NLH} when $\eps \to 0$, based on a small-noise asymptotic expansion formula for the non-singular, vector-valued $\Phi^4_1$-measure. Notably, this requires the treatment of polynomially growing functions $F$ in \eqref{asymexp}, which is a source of significant technical challenges in the present work. 
See \cite[(1.25), (1.26), (1.27)]{GT} for the proof of the synchronization by noise by exploiting the asymptotic expansion of the invariant measure.  In particular, the observable appearing in \cite[(1.25)]{GT} is not bounded, but of polynomial growth. \\
In Theorem \ref{THM:1}, we prove the small-noise asymptotic expansion for the \textit{singular}, scalar-valued $\Phi^4_2$-measure. Hence, by combining the argument in \cite{GT} with the asymptotic expansion in Theorem \ref{THM:1}, we expect that it will be possible to obtain an alternative approach to demonstrate synchronization by noise for the scalar-valued dynamical $\Phi^4_2$-model. This complements the earlier work \cite{GT0} by the first and third authors, which proved synchronization for the scalar-valued dynamical $\Phi^4_2$ model, based on comparison principles. The arguments introduced in the present work do not rely on comparison principles, and, therefore, potentially can be extended to the vector-valued setting, hence representing a stepping stone on the way to synchronization by noise for vector-valued, singular  $\Phi^4_2$-models.


\subsection{Organization of the paper} 
In Section \ref{SEC:Not}, we introduce some notations and preliminary lemmas. 
In Section \ref{SEC:Var}, we prove the quantified version of Varadhan's lemma (Proposition \ref{PROP:Var}) by exploiting the variational formulation of the partition
function (Lemma \ref{LEM:var3}). In Section \ref{SEC:ASYM}, we present an expansion of $\Phi^4_2$-measure around a new Gaussian measure and prove the main results (i.e.~Theorem \ref{THM:1}, \ref{THM:2}, and \ref{THM:3}).

\section{Notations and basic lemmas}
\label{SEC:Not}

When addressing regularities of functions and distributions, we  use $\eta > 0$ to denote a small constant. We usually  suppress the dependence on such $\eta > 0$ in estimates. For $a, b > 0$, $a\lesssim b$  means that
there exists $C>0$ such that $a \leq Cb$. By $a\sim b$, we mean that $a\lesssim b$ and $b \lesssim a$. Regarding space-time functions, we use the following short-hand notation $L^q_TL^r_x$ = $L^q([0, T]; L^r(\T^2))$, etc.

\subsection{Function spaces}
\label{SUBSEC:21}

Let $s \in \R$ and $1 \leq p \leq \infty$. We define the $L^p$-based Sobolev space $W^{s, p}(\T^d)$ by 
\begin{align*}
\| f \|_{W^{s, p}} = \big\| \F^{-1} [\jb{n}^s \ft f(n)] \big\|_{L^p}.
\end{align*}

\noi
When $p = 2$, we have $H^s(\T^d) = W^{s, 2}(\T^d)$.

Let $\phi:\R \to [0, 1]$ be a smooth  bump function supported on $[-\frac{8}{5}, \frac{8}{5}]$ 
and $\phi\equiv 1$ on $\big[-\frac 54, \frac 54\big]$.
For $\xi \in \R^d$, we set $\varphi_0(\xi) = \phi(|\xi|)$
and 
\begin{align}
\varphi_{j}(\xi) = \phi\big(\tfrac{|\xi|}{2^j}\big)-\phi\big(\tfrac{|\xi|}{2^{j-1}}\big)
\label{phi1}
\end{align}

\noi
for $j \in \N$.
Then, for $j \in \Z_{\geq 0} := \N \cup\{0\}$, 
we define  the Littlewood-Paley projector  $\pi_j$ 
as the Fourier multiplier operator with a symbol $\varphi_j$.
Note that we have 
\begin{align*}
\sum_{j = 0}^\infty \varphi_j (\xi) = 1
\end{align*}

\noi
for each $\xi \in \R^d$ and $f = \sum_{j = 0}^\infty \pi_j f$. We next recall the basic properties of the Besov spaces $B^s_{p, q}(\T^d)$
defined by the norm
\begin{equation*}
\| u \|_{B^s_{p,q}} = \Big\| 2^{s j} \| \pi_{j} u \|_{L^p_x} \Big\|_{\l^q_j(\Z_{\geq 0})}.
\end{equation*}

\noi
We denote the H\"older-Besov space by  $\mathcal{C}^s (\T^d)= B^s_{\infty,\infty}(\T^d)$.
Note that the parameter $s$ measures differentiability and $p$ measures integrability. In particular, $H^s (\T^d) = B^s_{2,2}(\T^d)$
and for $s > 0$ and not an integer, $\mathcal{C}^s(\T^d)$ coincides with the classical H\"older spaces $C^s(\T^d)$; see  \cite{Gra2}.

We recall the following basic estimates in Besov spaces, see \cite{BCD}, for example.

\begin{lemma}\label{LEM:KCKON0}
The following estimates hold.

\noi
\textup{(i) (interpolation)} 
Let $s, s_1, s_2 \in \R$ and $p, p_1, p_2 \in (1,\infty)$
such that $s = \ta s_1 + (1-\ta) s_2$ and $\frac 1p = \frac \ta{p_1} + \frac{1-\ta}{p_2}$
for some $0< \ta < 1$.
Then, we have
\begin{equation}
\| u \|_{W^{s,  p}} \les \| u \|_{W^{s_1, p_1}}^\ta \| u \|_{W^{s_2, p_2}}^{1-\ta}.
\label{INT0P}
\end{equation}

\noi
\textup{(ii) (embeddings)}
Let $s_1, s_2 \in \R$ and $p_1, p_2, q_1, q_2 \in [1,\infty]$.
Then, we have
\begin{align} 
\begin{split}
\| u \|_{B^{s_1}_{p_1,q_1}} 
&\les \| u \|_{B^{s_2}_{p_2, q_2}} 
\qquad \text{for $s_1 \leq s_2$, $p_1 \leq p_2$,  and $q_1 \geq q_2$},  \\
\| u \|_{B^{s_1}_{p_1,q_1}} 
&\les \| u \|_{B^{s_2}_{p_1, \infty}}
\qquad \text{for $s_1 < s_2$},\\
\| u \|_{B^0_{p_1, \infty}}
 &  \les  \| u \|_{L^{p_1}}
 \les \| u \|_{B^0_{p_1, 1}}.
\end{split}
\label{embed}
\end{align}

%


\smallskip

\noi
\textup{(iii) (Besov embedding)}
Let $1\leq p_2 \leq p_1 \leq \infty$, $q \in [1,\infty]$,  and  $s_2 \ge s_1 + d \big(\frac{1}{p_2} - \frac{1}{p_1}\big)$. Then, we have
\begin{equation}
 \| u \|_{B^{s_1}_{p_1,q}} \les \| u \|_{B^{s_2}_{p_2,q}}.
\label{BE0}
\end{equation}

\smallskip

\noi
\textup{(iv) (duality)}
Let $s \in \mathbb{R}$
and  $p, p', q, q' \in [1,\infty]$ such that $\frac1p + \frac1{p'} = \frac1q + \frac1{q'} = 1$. Then, we have
\begin{align}
\bigg| \int_{\T^d}  uv \, dx \bigg|
\le \| u \|_{B^{s}_{p,q}} \| v \|_{B^{-s}_{p',q'}},
\label{dual}
\end{align}

\noi
where $\int_{\T^d} u v \, dx$ denotes  the duality pairing between $B^{s}_{p,q}(\T^d)$ and $B^{-s}_{p',q'}(\T^d)$.

\smallskip
	
\noi		
\textup{(v) (fractional Leibniz rule)} 
Let $p, p_1, p_2, p_3, p_4 \in [1,\infty]$ such that 
$\frac1{p_1} + \frac1{p_2} 
= \frac1{p_3} + \frac1{p_4} = \frac 1p$. 
Then, for every $s>0$, we have
\begin{equation}
\| uv \|_{B^{s}_{p,q}} \les  \| u \|_{B^{s}_{p_1,q}}\| v \|_{L^{p_2}} + \| u \|_{L^{p_3}} \| v \|_{B^s_{p_4,q}} .
\label{fprod}
\end{equation}

\end{lemma}

\subsection{Tools from stochastic analysis}

We conclude this section by recalling some lemmas
from stochastic analysis.
See \cite{Bog, Shige} for basic definitions.
Let $(H, B, \mu)$ be an abstract Wiener space, that is, $\mu$ is a Gaussian measure on a separable Banach space $B$, and $H \subset B$ is its Cameron-Martin space.
Given  a complete orthonormal system $\{e_j \}_{ j \in \N} \subset B^*$ of $H^* = H$, 
we  define a polynomial chaos of order
$k$ to be an element of the form $\prod_{j = 1}^\infty H_{k_j}(\jb{x, e_j})$, 
where $x \in B$, $k_j \ne 0$ for only finitely many $j$'s, $k= \sum_{j = 1}^\infty k_j$, 
$H_{k_j}$ is the Hermite polynomial of degree $k_j$, 
and $\jb{\cdot, \cdot} = \vphantom{|}_B \jb{\cdot, \cdot}_{B^*}$ denotes the $B$--$B^*$ duality pairing.
We then 
denote the closure  of 
polynomial chaoses of order $k$ 
under $L^2(B, \mu)$ by $\mathcal{H}_k$.
The element in $\H_k$ 
is called homogeneous  Wiener chaos of order $k$.
We also set
\[ \H_{\leq k} = \bigoplus_{j = 0}^k \H_j\]

\noi
 for $k \in \N$.

Let $L = \Dl -x \cdot \nabla$ be 
 the Ornstein-Uhlenbeck operator. Then, 
it is known that 
any element in $\mathcal H_k$ 
is an eigenfunction of $L$ with eigenvalue $-k$.
Then, as a consequence
of the  hypercontractivity of the Ornstein-Uhlenbeck
semigroup $U(t) = e^{tL}$ due to Nelson \cite{Nelson}, 
we have the following Wiener chaos estimate
\cite[Theorem~I.22]{Simon}.

\begin{lemma}\label{LEM:hyp}
Let $k \in \N$.
Then, we have
\begin{equation*}
\|X \|_{L^p(\O)} \leq (p-1)^\frac{k}{2} \|X\|_{L^2(\O)}
 \end{equation*}
 
\noi
for any $p \geq 2$ and any $X \in \H_{\leq k}$.

\end{lemma}

We recall the following orthogonality
relation for the Hermite polynomials, 
see  \cite[Lemma 1.1.1]{Nua}.

\begin{lemma}\label{LEM:Wick2}
Let $X$ and $Y$ be jointly Gaussian random variables with mean zero 
and variances $\s_X$
and $\s_Y$.
Then, we have 
\begin{align*}
\E\big[ H_k(X; \s_X) H_\l(Y; \s_Y)\big] = \dl_{k\l} k! \big\{\E[ X Y] \big\}^k, 
\end{align*}

\noi
where $H_k (x,\s)$ denotes the Hermite polynomial of degree $k$ with variance parameter $\s$.

\end{lemma}

\section{Quantified version of Varadhan's lemma}
\label{SEC:Var}
In this section, we prove a quantified version of Varadhan's lemma (Proposition \ref{PROP:Var}), which enables us to  reduce the Gaussian functional integral to  neighborhoods of the manifold $\M$ of minimizers.

\subsection{Bou\'e-Dupuis variational formalism for the Gibbs measure}
\label{SUBSEC:var}

Let $W(t)$ be the cylindrical Wiener process on $L^2(\T^2)$ with respect to the underlying probability measure $\PP$, that is, 
\begin{align}
W(t)
 = \sum_{n \in \Z^2} B_n (t) e^{in\cdot x}, 
\label{W1}
\end{align}

\noi
where
$\{ B_n \}_{n \in \Z^3}$ is defined by 
$B_n(t) = \jb{\xi, \ind_{[0, t]} \cdot e^{in\cdot x}}_{ x, t}$.
Here, $\jb{\cdot, \cdot}_{x, t}$ denotes 
the duality pairing on $\T^2\times \R_{+}$. We  then define a centered Gaussian process $\<1>(t)$
by 
\begin{align}
\<1>(t)
=  \jb{\nabla}^{-1}W(t).
\label{P2}
\end{align}

\noi
Then, 
we have $\Law (\<1>(1)) = \mu$. 
By setting  $\<1>_N = \textbf{P}_N\<1> $, 
we have   $\Law (\<1>_N(1)) = (\textbf{P}_N)_\#\mu$. 
In particular, 
we have  $\E [\<1>_N(1)^2]  \sim \log N$.

Next, let $\Ha$ denote the space of drifts, which are the progressively measurable processes 
 belonging to
$L^2([0,1]; L^2(\T^2))$, $\PP$-almost surely. We next recall  the  Bou\'e-Dupuis variational formula \cite{BD, Ust}; in particular, see \cite[Theorem 7]{Ust} and \cite[Theorem 2]{BG}. See also \cite{OSeoT, Seong, RSTW, Seo1}.

\begin{lemma}\label{LEM:var3}
Let $\<1>(t)
=  \jb{\nabla}^{-1}W(t)$ be as in \eqref{P2}.
Fix $N \in \N$.
Suppose that  $F:C^\infty(\T^2) \to \R$
is measurable such that $\E\big[|F(\<1>_N(1))|^p\big] < \infty$
and $\E\big[|e^{-F(\<1>_N(1))}|^q \big] < \infty$ for some $1 < p, q < \infty$ with $\frac 1p + \frac 1q = 1$.
Then, we have
\begin{align}
- \log \E\Big[e^{-F(\<1>_N(1))}\Big]
= \inf_{\dr \in \mathbb H_a}
\E\bigg[ F( \<1>_N(1) +  \Dr_N(1)) + \frac{1}{2} \int_0^1 \| \dr(t) \|_{L^2_x}^2 dt \bigg], 
\label{BD1}
\end{align}

\noi
where $\Dr(t)$ is defined by 
\begin{align}
\Dr(t) = \int_0^t \jb{\nabla}^{-1} \dr(t') dt'.
\label{P3a}
\end{align}
\end{lemma}

We define the second, third, and fourth Wick powers of $\<1>_N(t)$ as follows
\begin{align}
\<2>_N(t)&=\<1>_N^2(t)-\<tadpole>_N(t) \notag \\
\<3>_N(t)&=\<1>_N^3(t)-3\<tadpole>_N(t) \notag \\
\<4>_N(t)&=\<1>_N^4(t)- 6 \<tadpole>_N(t) \<1>_N^2(t)+3\<tadpole>_N^2(t)
\label{Wick0}
\end{align}

\noi
where $\<tadpole>_N(t)=\E_\PP\big[| \<1>_N(t)|^2 \big] \sim t \log N$. We next recall  a lemma on pathwise regularity estimates  of $(\<1>(t), \<2>(t), \<3>(t), \<4>(t))$ and $\Dr(t)$. 


\begin{lemma}  \label{LEM:Cor00}

\textup{(i)} 
For any finite $p \ge 2$, $t\in [0,1]$, and $\eta>0$, each Wick power in \eqref{Wick0} converges to the limit in $L^p(\O; \mathcal{C}^{-\eta}(\T^2))$ as $N\to \infty$  and also almost surely in $\mathcal{C}^{-\eta}(\T^2)$. Moreover, we have
\begin{equation}
\begin{split}
\E \Big[& \|\<1>_{N}(t)\|_{\mathcal{C}^{ - \eta}}^p+  \| \<2>_{N}(t) \|_{\mathcal{C}^{ - \eta}}^p+ \| \<3>_{N}(t)   \|_{\mathcal{C}^{-\eta}}^p \Big] \les  p^{\frac 32}   <\infty
\end{split}
\label{QS0}
\end{equation}

\noi
uniformly in $N \in \N\cup\{\infty\}$\footnote{When $N=\infty$, it can be understood as the identity operator.} and $t \in [0, 1]$.

\smallskip

\noi
\textup{(ii)} 
For any $N \in \N\cup \{\infty\}$ and $t\in [0,1]$, we have 
\begin{align}
\E \bigg[ \int_{\T^2} \<2>_{N}(1) dx \bigg]
&= 0,\\
\E \bigg[ \int_{\T^2} \<3>_{N}(1) dx \bigg]
&= 0,\\
\E \bigg[ \int_{\T^2} \<4>_{N}(1) dx \bigg]
&= 0.
\end{align}

\smallskip

\noi
\textup{(iii)}
The drift term $ \dr\in \mathbb{H}_a $ has the regularity
of the Cameron-Martin space, that is, for any $\dr\in \mathbb{H}_a $, we have
\begin{align}
\| \Dr(1)\|_{ H^{1}_x}^2 \leq \int_0^1 \|  \dot{\Dr}(t) \|_{H^1_x}^2dt.
\label{CCS5}
\end{align}

\noi
where $\dot{\Dr}(t)=\jb{\nb}^{-1}\dr(t)$.
\end{lemma}

\begin{proof}
Part (i) and (ii) of Lemma \ref{LEM:Cor00} follow from a standard computation and thus we omit details. See, for example, \cite[Proposition 2.3]{OTh2} and \cite[Proposition 2.1]{GKO0}. As for Part (iii), the estimate \eqref{CCS5} follows from Minkowski’s and Cauchy-Schwarz’ inequalities. See also the proof of Lemma 4.7 in \cite{GOTW} .
\end{proof}

In the following, for simplicity, we denote $\<1>_N(1), \<2>_N(1), \<3>_N(1), \<4>_N(1)$ and $\Dr_N(1)$ by $\<1>_N, \<2>_N, \<3>_N, \<4>_N$, and $\Dr_N$, respectively.

\subsection{Concentration of the $\Phi^4_2$-measure}

In this subsection, we present the proof of a quantified version of Varadhan's lemma by
relying on the Bou\'e-Dupuis formula (Proposition \ref{PROP:Var}).

We first study the optimizers for the specified Hamiltonian 
\begin{align}
H(\phi)=\frac 12\int_{\T^2} |\nb \phi|^2 dx+\frac{\ld}{4} \int_{\T^2} 
\big(|\phi|^2-1\big)^2 dx \label{H2},
\end{align}

\noi
along with its stability property (Lemma \ref{LEM:sta}), where $\ld >0$ measures the strength of the interaction potential\footnote{Since in the subsequent analysis the precise value of $\ld>0 $ does not play any role, we set $\ld = 1$ in the following.}. In particular, we have $\inf_{\phi \in H^1(\T^2)} H(\phi) \ge 0$. We note that the minimum occurs at either $\phi=+1$ or $-1$ and thus $\min_{\phi \in H^1(\T^2)}H(\phi)=0$.

We first study the stability property of minimizers. For future use, let $\M$ be the set of minimum points of $V(u)=\frac 14(|u|^2-1)^2$.

\begin{lemma}\label{LEM:sta}
Given $\dl>0$, there exists $c=c(\dl)>0$ such that for any $\phi \in H^1(\T^2)$ satisfying $\inf_{w \in \M} \|\phi- w \|_{H^1} \ge \dl $, we have
\begin{align}
H(\phi)\geq \inf_{\varphi \in H^1} H(\varphi)+c(\dl)=c(\dl).
\label{sta01}
\end{align}

\end{lemma}

\begin{proof}
In order to prove \eqref{sta01}, we proceed by contradiction. Namely, we
suppose that there exists $\dl>0$ such that for every $n\in \N$, there exists $\varphi_n$ with
\begin{align}
\inf_{w\in \M} \| \varphi_n- w \|_{H^1} > \dl
\label{cont}
\end{align}

\noi
but 
\begin{align}
\inf_{\varphi \in H^1} H(\varphi)\ \le H(\varphi_n)<\inf_{\varphi \in H^1} H(\varphi)+\frac 1n,
\label{minimiz}
\end{align}

\noi
which implies that $\{\varphi_n\}$ is a minimizing sequence. Hence, $\{\varphi_n\}$ is bounded in $H^1(\T^2)$ and so thanks to Rellich's Theorem, for all $\eta>0$, there exists $\psi $ in $H^{1-\eta}(\T^2)$ such that up to subsequences, we have
\begin{align}
\|\varphi_n-\psi \|_{H^{1-\eta}} \to 0.
\label{Relli}
\end{align}

\noi
Since $\lim_{n \to \infty}H(\varphi_n)=0$, by taking a further subsequence, we have
\begin{align*}
0=\lim_{n\to \infty }\int_{\T^2} \big(|\varphi_n|^2-1\big)^2 dx \ge \int_{\T^2} \lim_{n\to \infty } \big(|\varphi_n|^2-1\big)^2 dx =\int_{\T^2} \big(|\psi|^2-1 \big)^2 dx\ge 0,
\end{align*}

\noi
which implies that $\psi=1$ or $-1$ almost everywhere. 
Moreover, thanks to $\lim_{n\to \infty} \| \nb \varphi_n \|_{L^2}=0=\| \nb \psi \|_{L^2}$ and \eqref{Relli}, we have 
\begin{align*}
\|\varphi_n-\psi \|_{H^1} \to 0,
\end{align*}
 
\noi
from which $\psi$ becomes a minimizer of $H$. Therefore, there exists $w_0 \in \M$ such that $\psi=w_0$ and
\begin{align*}
0\leftarrow \|\varphi_n-w_0 \|_{H^1} \ge \inf_{w\in \M} \| \varphi_n- w \|_{H^1} > \dl,
\end{align*}

\noi
which makes a contradiction.

\end{proof}

\begin{remark}\rm
The argument in the proof of Lemma \ref{LEM:sta} shows that if $u$ is a minimizer of the Hamiltonian \eqref{H2}, then either $u=1$ or $u=-1$. Indeed, one can set $\varphi_n=u$ for every $n\ge 1$ and so $\varphi_n$ becomes a minimizing sequence. Hence,
we can proceed with the argument. 
\end{remark}



We are now ready to prove the quantified version of Varadhan's lemma.

\begin{proposition}\label{PROP:Var}
Let $\dl>0$ and $F: \mathcal{C}^{-\eta}(\T^2;\R)\to \R$ be a function with at most polynomial growth. Then, there exists $c=c(\dl)>0$ and sufficiently small $\eps_0=\eps_0(\dl)>0$ such that 
\begin{align}
\sup_{\kappa \ge 1}\int_{\{ \textup{dist}_{\mathcal{C}^{-\eta} }(\phi,\mathcal{M}) \ge \dl  \}   } F(\phi) \exp\bigg\{ -\frac 1\eps \mathcal{V}(\phi)\bigg\} \mu_{\eps}^\kappa(d\phi) \les \exp\Big\{ -\frac c{2\eps} \Big\}
\label{VARLEM2}
\end{align}

\noi
for every $0<\eps \le \eps_0$, where $\mu_\eps^\kappa$ is the Gaussian measure with covariance $\eps(-\kappa\Dl+1)^{-1}$ and $\mathcal{V}(\phi)=\mathcal{V}_1(\phi)+\mathcal{V}_2(\phi)$, with
\begin{align}
\mathcal{V}_1(\phi):&=\frac 14 \int_{\T^2} :\! \phi^4  \!: dx  -\frac 12 \int_{\T^2} :\! \phi^2  \!:  dx +\frac 14   \notag \\
\mathcal{V}_2(\phi):&=-\frac 1{2} \int_{\T^2} :\! \phi^2 \!: dx \label{V12},
\end{align}

\noi
and $\M$ is the set of minimum points of $V(u)=\frac 14(|u|^2-1)^2$.   
\end{proposition}

\begin{remark}\rm
In comparison to the proof presented in \cite[Lemma 2.5]{GT}, it is important to note that the coercivity of the potential $\mathcal{V}$ cannot be applied. Namely, the renormalization process destroys the coercive structure of the potential.
\end{remark}

\begin{remark}\rm
The error term $O(e^{-\frac{c}{2\eps}})$ in \eqref{VARLEM2} is uniform in the viscosity coefficient $\kappa$ if $\kappa \ge \kappa_0 > 0$ for arbitrary fixed $\kappa_0 > 0$. For simplicity of notation, we set $\kappa_0 = 1$ in the statement, but the proof holds for every $\kappa_0 > 0$.
\end{remark}

\begin{proof}

To simplify the notation, we let $\kappa = 1$ but the proofs can be extended trivially to arbitrary $\kappa\ge \kappa_0 > 0$. We first observe that if $u$ represents a Gaussian random variable with $\Law(u)=\mu$, applying the linear transformation $u \mapsto \sqrt{\eps} u$, $\sqrt{\eps} u$ yields a
Gaussian random variable with $\Law(\sqrt{\eps}u )=\mu_\eps$. Hence, we get
 \begin{align}
&\int_{  \{   \text{dist}_{C^{-\eta}}(\phi, \M ) \ge \dl    \}     } F( \phi)\exp\bigg\{ -\frac{1}{\eps}\mathcal{V}(\phi) \bigg\} d\mu_\eps(\phi) \notag \\
&=\int_{  \{   \text{dist}_{C^{-\eta}}(\sqrt{\eps}\phi, \M ) \ge \dl    \}   \}    } F(\sqrt{\eps} \phi)\exp\bigg\{ -\frac{1}{\eps}\mathcal{V}(\sqrt{\eps}\phi) \bigg\} d\mu(\phi). \notag 
\end{align}

\noi
We initially split the integral in the following manner
\begin{align}
&\int_{  \{   \text{dist}_{C^{-\eta}}(\sqrt{\eps}\phi, \M ) \ge \dl    \}    } F(\sqrt{\eps} \phi)\exp\bigg\{ -\frac{1}{\eps}\mathcal{V}(\sqrt{\eps}\phi) \bigg\} d\mu(\phi)  \notag \\
&=\int_{  \{   \text{dist}_{C^{-\eta}}(\sqrt{\eps}\phi, \M ) \ge \dl    \} \cap  \{\|\sqrt{\eps}\phi \|_{C^{-\eta}}\le M  \}    } F(\sqrt{\eps} \phi)\exp\bigg\{ -\frac{1}{\eps}\mathcal{V}(\sqrt{\eps}\phi) \bigg\} d\mu(\phi)  \notag  \\
&\hphantom{X}+ \int_{  \{   \text{dist}_{C^{-\eta}}(\sqrt{\eps}\phi, \M ) \ge \dl    \} \cap  \{\|\sqrt{\eps}\phi \|_{C^{-\eta}}> M  \}    } F(\sqrt{\eps} \phi)\exp\bigg\{ -\frac{1}{\eps}\mathcal{V}(\sqrt{\eps}\phi) \bigg\} d\mu(\phi), 
\label{III}
\end{align}

\noi
where $M \ge 1$ will be specified later. Next, we consider the first term in \eqref{III}.
For convenience, we omit the topology $C^{-\eta}$ in $\text{dist}_{C^{-\eta}}$ in the following. Thanks to the polynomial growth of $F$, we have
\begin{align}
&\int_{  \{   \text{dist}(\sqrt{\eps}\phi, \M ) \ge \dl    \} \cap  \{\|\sqrt{\eps}\phi \|_{C^{-\eta}}\le M  \}    } F(\sqrt{\eps} \phi)\exp\bigg\{ -\frac{1}{\eps}\mathcal{V}(\sqrt{\eps}\phi) \bigg\} d\mu(\phi) \notag \\
& \le M^k \int \exp\bigg\{ -\frac{1}{\eps}\mathcal{V}(\sqrt{\eps}\phi)  \ind_{\{
\text{dist}(\sqrt{\eps}\phi, \M ) \ge \dl  \}   } \bigg\} d\mu(\phi)=M^k \wt Z_\eps \notag \\
&=M^k\lim_{N\to \infty} \wt Z_{N,\eps},
\label{limpart}
\end{align}

\noi
where $k$ comes from the polynomial growth of $F$ and  in the last line we used Proposition \ref{PROP:NEL} to guarantee the convergence and
\begin{align*}
\wt Z_{N,\eps}:&=\int \exp\bigg\{ -\frac{1}{\eps}\mathcal{V}(\sqrt{\eps}\phi_N)  \ind_{\{
\text{dist}(\sqrt{\eps}\phi_N, \M ) \ge \dl  \}   } \bigg\} d\mu(\phi)\\
&=\E_{\PP}\Bigg[\exp\bigg\{ -\frac{1}{\eps}\mathcal{V}(\sqrt{\eps}\<1>_N)  \ind_{\{
\text{dist}(\sqrt{\eps}\<1>_N, \M ) \ge \dl  \}   } \bigg\} 
\Bigg].
\end{align*}


\noi
By exploiting the Bou\'e-Dupuis formula (Lemma \ref{LEM:var3}), we have
\begin{align}
-\log \wt Z_{N,\eps}&=\inf_{\dr \in \Ha }\E\bigg[  \frac{1}{\eps}\mathcal{V}(\sqrt{\eps}(\<1>_N+\Dr_N) ) \ind_{\{ \text{dist}(\sqrt{\eps}\<1>_N+\sqrt{\eps}\Dr_N, \M ) \ge \dl  \}   }   +\frac 12 \int_0^1 \| \dot{\Dr}(t) \|_{H^1_x}^2 dt   \bigg] \notag  \\
&=\inf_{\dr \in \Ha }\E\bigg[  \Phi(\Xi_N^\eps, \sqrt{\eps} \Dr_N)  \ind_{\{
\text{dist}(\sqrt{\eps}\<1>_N+\sqrt{\eps}\Dr_N, \M ) \ge \dl  \}   }   +\frac 12 \int_0^1 \| \dot{\Dr}(t) \|_{H^1_x}^2 dt   \bigg],
\label{CCZ0}
\end{align}

\noi
where $\Xi_N^{\eps}=\big(\eps^2\<4>_N, \eps^{\frac 32}\<3>_N, \eps \<2>_N, \sqrt{\eps} \<1>_N \big)$ is the enhanced data set and
\begin{align*}
\Phi(\Xi_N^\eps, \sqrt{\eps} \Dr_N)&=\frac \eps4\int_{\T^2} \<4>_N \,dx + \frac {3\eps}2\int_{\T^2} \<3>_N\, \Dr_N dx +\frac {3\eps}2 \int_{\T^2  } \<2>_N\, \Dr_N^2 dx +\frac {3\eps}{4}\int_{\T^2} \<1>_N \, \Dr_N^3 dx \notag \\
& \hphantom{X}+\frac{\eps}{4} \int_{\T^2} \Dr_N^4 dx-\int_{\T^2} \<2>_N dx-2\int_{\T^2} \<1>_N  \Dr_N dx-\int_{\T^2} \Dr_N^2 dx+\frac 1{4\eps}.
\end{align*}

\noi
Thanks to Lemma \ref{LEM:Dr1} and Lemma \ref{LEM:Cor00}, we have that for every $\zeta>0$ and large enough $p$ 
\begin{align}
&\bigg|\Phi(\Xi_N^\eps, \sqrt{\eps} \Dr_N)-\frac{\eps}{4}\int_{\T^2} \Dr_N^4 dx-\frac{1}{4\eps}\bigg| \notag \\
&  \le  \zeta \frac\eps4 \int_{\T^2} \Dr_N^4 dx+\zeta \| \Dr_N \|^2_{H^1}+\int_{\T^2} \Dr_N^2 dx+O(\|\Xi_N^\eps\|^p)+O(\zeta^{-1}),
\label{CCZ1}
\end{align}

\noi
where $\VVert{\Xi_N^\eps}^p:=\|\eps^2\<4>_N \|_{C^{-\eta}}^p+\|\eps^{\frac 32}\<3>_N \|_{C^{-\eta}}^p+\|\eps \<2>_N \|_{C^{-\eta}}^p+\|\sqrt{\eps} \<1>_N \|_{C^{-\eta}}^p$. By combining \eqref{CCZ0} and \eqref{CCZ1}, we have that for sufficiently small $\zeta>0$
\begin{align}
&\log  \wt Z_{N, \eps} \notag \\
&\le \sup_{\Dr \in \mathcal{H}^1 }\E\bigg[- \Phi(\Xi_N^\eps,  \sqrt{\eps} \Dr_N)  \ind_{\{
\text{dist}(\sqrt{\eps}\<1>_N+\sqrt{\eps}\Dr_N, \M ) \ge \dl  \}   } -\frac{1}{2}\|\Dr \|_{H^1}^2  \bigg]  \notag \\
&\le \sup_{\Dr \in \mathcal{H}^1 }\E\bigg[- \Big( \Phi(\Xi_N^\eps, \sqrt{\eps}  \Dr_N)  +\frac{1}{2} \|  \Dr_N\|_{H^1}^2 \Big) \ind_{\{\text{dist}(\sqrt{\eps}\<1>_N+\sqrt{\eps}\Dr_N, \M ) \ge \dl  \}   } \bigg] \notag \\
&\le \sup_{\Dr \in \mathcal{H}^1 }\E\Bigg[- \bigg( \Big(\frac 14 -\zeta \Big)\eps \int_{\T^2} \Dr_N^4 dx+\Big(\frac 12 -\zeta \Big)\int_{\T^2} |\nb \Dr_N|^2 dx -\Big(\frac 12 +\zeta \Big)\int_{\T^2}  \Dr_N^2 dx+\frac{1}{4\eps} \bigg) \notag \\
&\hphantom{XXXXXXXXX}\times  \ind_{\{
\text{dist}(\sqrt{\eps}\<1>_N+\sqrt{\eps }\Dr_N, \M ) \ge \dl  \}   }  \Bigg]+O(1),
\label{Ia}
\end{align}

\noi
where $\H^1$ represents the collection of drifts $\Dr$, characterized as processes that belong to $H^1(\T^2)$ $\PP$-almost surely (possibly non-adapted) and $O(1)$ arises from Lemma \ref{LEM:Cor00} (i) by computing the expected values of the higher moments for each component of $\Xi_N^{\eps}=\big(\eps^2\<4>_N, \eps^{\frac 32}\<3>_N, \eps \<2>_N, \sqrt{\eps} \<1>_N \big)$ in $C^{-\eta}$, uniformly in $N\ge 1$ and $O(\zeta^{-1})$ in \eqref{CCZ1}. From \eqref{Ia} and  using the change of variable $\Dr \to  \sqrt{\eps}^{-1} \Dr$, we have
\begin{align}
&\log \wt Z_{N,\eps } \notag \\
& \le  \sup_{\Dr \in \mathcal{H}^1 }\E\Bigg[- \bigg( \Big(\frac 14 -\zeta \Big)\frac{1}{\eps} \int_{\T^2} \Dr_N^4 dx+\Big(\frac 12 -\zeta \Big)\frac{1}{\eps} \int_{\T^2} |\nb \Dr_N|^2 dx \notag \\
&\hphantom{XXXXXXXX}-\Big(\frac 12 +\zeta \Big)\frac{1}{\eps}\int_{\T^2}  \Dr_N^2 dx+\frac{1}{4\eps}\bigg) \ind_{\{
\text{dist}(\sqrt{\eps}\<1>_N+\Dr_N, \M ) \ge \dl  \}   }  \Bigg]+O(1) \notag \\
&\le \sup_{\Dr \in \mathcal{H}^1 } \E\Bigg[  -\bigg(  \frac{1}{4\eps}\int_{\T^2} \big(|\Dr_N|^2-1\big)^2 dx -\frac{\zeta}{\eps}\int_{\T^2}\Dr_N^4 dx-\frac{\zeta}{\eps} \int_{\T^2} \Dr_N^2 dx \notag \\
&\hphantom{XXXXXXXX}   +\Big(\frac 12 -\zeta \Big)\frac{1}{\eps} \int_{\T^2} |\nb \Dr_N|^2 dx    \bigg) \ind_{\{
\text{dist}(\sqrt{\eps}\<1>_N+\Dr_N, \M ) \ge \dl  \}   }      \Bigg] +O(1).
\label{Ia2}
\end{align}

\noi
There exists a sufficiently small $\zeta_0>0$ and a large $M_0>0$ such that for every $0< \zeta\le \zeta_0$, we have that $\frac{1}{8}(|a|^2-1)^2 \ge \zeta (|a|^4+|a|^2)$ holds for every $|a| \ge M_0$, which is equivalent to write $\frac{1}{4}(|a|^2-1)^2-\zeta |a|^4-\zeta|a|^2\ge \frac{1}{8}(|a|^2-1)^2$. Hence, we get
\begin{align}
&\frac{1}{4\eps}\int_{\{|\Dr_N|\ge M_0\}}  \big(|\Dr_N|^2-1\big)^2 dx -\frac{\zeta}{\eps}\int_{\{|\Dr_N|\ge M_0\}}\Dr_N^4 dx-\frac{\zeta}{\eps} \int_{\{|\Dr_N|\ge M_0\}} \Dr_N^2 dx \notag \\
&\ge \frac{1}{8\eps}\int_{\{|\Dr_N|\ge M_0\}} \big(|\Dr_N|^2-1\big)^2 dx.
\label{Ia2022}
\end{align}

\noi
Moreover, we write
\begin{align}
&\frac{1}{4\eps}\int_{\{|\Dr_N|\le M_0\}}  \big(|\Dr_N|^2-1\big)^2 dx -\frac{\zeta}{\eps}\int_{\{|\Dr_N|\le M_0\}}\Dr_N^4 dx-\frac{\zeta}{\eps} \int_{\{|\Dr_N|\le M_0\}} \Dr_N^2 dx \notag\\
& \ge \frac{1}{8\eps}\int_{\big\{ |\Dr_N| \le  M_0 \big\}} \big(|\Dr_N|^2-1\big)^2dx-\frac{c_1 \zeta}{\eps}
\label{Ia201}
\end{align}

\noi
for some constant $c_1>0$.
Hence, by combining \eqref{Ia2}, \eqref{Ia2022}, and \eqref{Ia201}, we have
\begin{align}
\log \wt Z_{N,\eps} \le \sup_{\Dr \in \mathcal{H}^1 }\E\bigg[-\frac{c_2}{\eps}H(\Dr_N)\ind_{\{
\text{dist}(\sqrt{\eps}\<1>_N+\Dr_N, \M ) \ge \dl  \}   }   \bigg]+\frac{c_1\zeta}{\eps}+O(1)
\label{Ia202}
\end{align}

\noi
for some constant $c_2>0$, where
\begin{align*}
H(\Dr_N)=\frac 12 \int_{\T^2 } |\nb \Dr_N|^2 dx+\frac 14\int_{\T^2} \big(|\Dr_N|^2-1\big)^2 dx.
\end{align*}

\noi
Also, for sufficiently large $A \ge 1$, there exists $\eps_0=\eps_0(A,\dl)>0$ such that for all $0 < \eps \le \eps_0$,
\begin{align}
\big\{  \| \Dr_N-w \|_{C^{-\eta}} \ge 2 \dl \}&\cap \big\{ \| \<1>_N\|_{C^{-\eta}} \le A \big\} \subset \big\{  \|\sqrt{\eps}\<1>_N+\Dr_N-w \|_{C^{-\eta}} \ge \dl \big\} 
\label{set1}
\end{align}

\noi
for any $w\in \M=\{-1,+1\}$. Hence, it follows from \eqref{Ia202}, \eqref{set1}, and Lemma \ref{LEM:sta} that 
\begin{align}
\log \wt Z_{N,\eps} &\le  \sup_{\Dr \in \mathcal{H}^1 }  \E\bigg[ -\frac{c}{\eps}H(\Dr_N) \ind_{ \{ \text{dist}(\Dr_N,\M) \ge 2\dl   \}\cap \{ \| \<1>_N\|_{C^{-\eta}} \le A  \} }     \bigg]+\frac{c_1\zeta}{\eps}+O(1) \notag \\
& \le - \frac{c}{\eps} \E\bigg[  \inf_{\Dr \in \H^1 } H(\Dr_N)  \ind_{ \{ \text{dist}(\Dr_N,\M) \ge 2\dl   \}\cap \{ \| \<1>_N\|_{C^{-\eta}} \le A  \} }      \bigg]+\frac{c_1\zeta}{\eps} +O(1) \notag \\
& \le -\frac{c(\dl)}{\eps}+\frac{c_1\zeta}{\eps}+O(1),
\label{sta1a}
\end{align}

\noi
uniformly in $N \ge 1$, where $c(\dl)$ comes from Lemma \ref{LEM:sta}. This implies that choosing $\zeta$ small enough gives
\begin{align}
\wt Z_{N,\eps} \les \exp \Big\{-\frac{c(\dl)}{2\eps} \Big\}.
\label{CoPh00}
\end{align}

\noi
Hence, from \eqref{limpart} and \eqref{CoPh00}, we have that there exists sufficiently small $\eps_0=\eps_0(\dl)>0$ such that 
\begin{align}
&\int_{  \{   \text{dist}(\sqrt{\eps}\phi, \M ) \ge \dl    \} \cap  \{\|\sqrt{\eps}\phi \|\le M  \}    } F(\sqrt{\eps} \phi)\exp\bigg\{ -\frac{1}{\eps}\mathcal{V}(\sqrt{\eps}\phi) \bigg\} d\mu(\phi) \notag \\
&\les M^k\exp\Big\{-\frac{c(\dl)}{2\eps}\Big\} \les \exp\Big\{-\frac{c(\dl)}{4\eps}\Big\}
\label{I0}
\end{align}

\noi
for all $0<\eps\le  \eps_0$, where $M$ will be specified later (see \eqref{s2} and \eqref{G22}).

Now, let us next examine the second term in equation \eqref{III}. In the subsequent discussion, we establish the existence of a sufficiently large constant $M \ge 1$ such that
\begin{align*}
\int_{  \{   \text{dist}(\sqrt{\eps}\phi, \M ) \ge \dl    \} \cap  \{\|\sqrt{\eps}\phi \|> M  \}    } F(\sqrt{\eps} \phi)\exp\bigg\{ -\frac{1}{\eps}\mathcal{V}(\sqrt{\eps}\phi) \bigg\} d\mu(\phi) \les  \exp\Big\{ -\frac {c'}{2\eps} \Big\}.
\end{align*}

\noi
Given a sufficiently large $K=K(\eps) \ge 1$ which will be specified later (see \eqref{II}), let $G$ be a bounded smooth non-negative function such that 
\begin{align}
G(\phi)= 
\begin{cases}
K, & \quad \text{if } \quad \|\phi\|_{C^{-\eta}} \leq \frac M2,\\
0, & \quad \text{if } \quad \|\phi\|_{C^{-\eta} }> M.
\end{cases}
\label{T0a}
\end{align}

\noi
The existence of such $G$ is guaranteed by Uryshon’s lemma, as both sets $\{\|\phi \|_{C^{-\eta} } \le \frac{M}{2}  \}$ and $\{\|\phi \|_{C^{-\eta} } \ge M  \}$  are closed. We also note that thanks to the polynomial growth of $F$, we get 
\begin{align}
F(\phi) \les \| \phi\|_{C^{-\eta}}^k \les \exp\{\kappa \|\phi \|_{C^{-\eta}}^2  \}
\label{poa}
\end{align}

\noi
for some $k\ge 1$ and small $\kappa>0$. Hence, by combining \eqref{T0a} and \eqref{poa}, we have
\begin{align}
& \int_{  \{   \text{dist}(\sqrt{\eps}\phi, \M ) \ge \dl    \} \cap  \{\|\sqrt{\eps}\phi \|> M  \}    } F(\sqrt{\eps} \phi)\exp\bigg\{ -\frac{1}{\eps}\mathcal{V}(\sqrt{\eps}\phi) \bigg\} d\mu(\phi)   \notag \\
&\les \int  \exp\bigg\{ -G(\sqrt{\eps}\phi) +\kappa\| \sqrt{\eps}  \phi\|_{C^{-\eta}}^2 -\frac 1\eps \mathcal{V}(\sqrt{\eps}\phi)   \bigg\}=Z_{\eps, G} \notag \\
&=\lim_{N\to \infty} Z_{N,\eps, G},
\label{wtZ}
\end{align}

\noi
where in the last line we used Proposition \ref{PROP:NEL} and
\begin{align*}
Z_{N,\eps,G}:=\int  \exp\bigg\{ -G(\sqrt{\eps}\phi_N) +\kappa\| \sqrt{\eps}  \phi_N\|_{C^{-\eta}}^2 -\frac 1\eps \mathcal{V}(\phi_N)   \bigg\}.
\end{align*}

\noi
Then, by using the Bou\'e-Dupuis variational formula (Lemma \ref{LEM:var3}), we have
\begin{align}
-\log Z_{N,\eps, G}&=\inf_{\dr \in \Ha}\E \bigg[ G(\sqrt{\eps}\<1>_N+\sqrt{\eps} \Dr_N) +\frac 1\eps \mathcal{V}(  \sqrt{\eps}\<1>_N +\sqrt{\eps}\Dr_N )-\kappa \| \sqrt{\eps}\<1>_N +\sqrt{\eps} \Dr_N \|_{C^{-\eta}}^2 \notag\\
&\hphantom{XXXX}  +\frac 12 \int_0^1 \|\nb \dot{\Dr}(t) \|_{L^2_x}^2 dt+\frac 12 \int_0^1 \| \dot{\Dr}(t) \|_{L^2_x}^2 dt  \bigg].
\label{wtZ2}
\end{align}

\noi
It follows from Sobolev's inequality, Chebyshev's inequality and choosing $M\gg 1$ that 
\begin{align}
\PP\bigg\{ \| \sqrt{\eps}\<1>_N+\sqrt{\eps}\Dr_N \|_{C^{-\eta}}>\frac M2 \bigg\}&\le \PP\bigg\{ \|  \sqrt{\eps}\<1>_N\|_{C^{-\eta}} >\frac M4 \bigg\}+ \PP\bigg\{  \| \sqrt{\eps}\Dr_N  \|_{H^1}   >\frac M4 \bigg\} \notag \\
&\le \frac 12 +\frac {16\eps}{M^2} \E \Big[\| \Dr_N \|_{H^1_x}^2  \Big].
\label{s2}
\end{align}

\noi
Then, from \eqref{T0a}, \eqref{s2} and Lemma \ref{LEM:Cor00}, we have
\begin{align}
\E\bigg[ G(\sqrt{\eps}\<1>_N+\Dr_N) \bigg]&\ge K \E\Big[\ind_{ \big\{ \|\sqrt{\eps}\<1>_N+\sqrt{\eps}\Dr_N   \|_{C^{-\eta}} \le \frac M2   \big\} } \Big] \notag \\
&\ge \frac K2-\frac {16 K \eps}{M^2} \E \Big[ \| \Dr_N \|_{H^1_x}^2  \Big] \notag \\
&\ge \frac K2-\frac 14\E\bigg[ \int_0^1 \| \dot{\Dr}(t)\|_{H^1_x}^2 dt \bigg],
\label{G22}
\end{align}

\noi
where we take sufficiently large $M$. Hence, by combining \eqref{wtZ2} and \eqref{G22}, we have 
\begin{align}
-\log Z_{N,\eps, G} &\ge \frac K2+\inf_{\dr \in \Ha}\E\bigg[ \frac 1\eps \mathcal{V}(\sqrt{\eps}\<1>_N +\sqrt{\eps}\Dr_N)-\kappa \| \sqrt{\eps} \<1>_N +\sqrt{\eps}\Dr_N  \|_{C^{-\eta}}^2  +\frac{1}{4} \int_0^1 \| \dot{\Dr}(t)\|_{H^1_x}^2   dt\bigg]   \notag \\ 
&\ge \frac K2 + \inf_{\dr \in \Ha}\E\bigg[  \Phi(\Xi_N^\eps, \sqrt{\eps} \Dr_N)-\kappa \| \sqrt{\eps} \<1>_N +\sqrt{\eps}\Dr_N  \|_{C^{-\eta}}^2   +\frac 14 \int_0^1 \| \dot{\Dr}(t) \|_{H^1_x}^2 dt \bigg],
\label{logwtZ}
\end{align}

\noi
where
\begin{align}
\Phi(\Xi_N^\eps, \sqrt{\eps} \Dr_N)&=\frac \eps4\int_{\T^2} \<4>_N \,dx + \frac {3\eps}2\int_{\T^2} \<3>_N\, \Dr_N dx +\frac {3\eps}2 \int_{\T^2  } \<2>_N\, \Dr_N dx +\frac {3\eps}{4}\int_{\T^2} \<1>_N \, \Dr_N^3 dx \notag \\
& \hphantom{X}+\frac{\eps}{4} \int_{\T^2} \Dr_N^4 dx-\int_{\T^2} \<2>_N dx-2\int_{\T^2} \<1>_N  \Dr_N dx-\int_{\T^2} \Dr_N^2 dx+\frac 1{4\eps}.
\label{pathlow}
\end{align}

\noi
The main strategy is to establish a pathwise lower bound on $\Phi(\Xi_N^\eps, \sqrt{\eps} \Dr_N)$ in~\eqref{pathlow}, uniformly in $\dr \in \Ha$ and $N\ge 1$, by making use of the positive terms
\begin{equation}
\U(\Dr_N) = \frac{\eps}{4} \int_{\T^2} \Dr_N^4 dx +\frac 14 \int_0^1 \| \dot{\Dr}(t) \|_{H^1_x}^2 dt  +\frac 1{4\eps}  
\label{UN0}
\end{equation}

\noi
in \eqref{logwtZ}. From \eqref{pathlow}, Lemma \ref{LEM:Dr1}, and Lemma \ref{LEM:Cor00}, we have that for every $\zeta>0$
\begin{align}
&\E\Bigg[\bigg|\Phi(\Xi_N^\eps, \sqrt{\eps} \Dr_N)-\frac{\eps}{4}\int_{\T^2} \Dr_N^4 dx-\frac{1}{4\eps}\bigg|\Bigg] \notag \\
&  \le  \zeta \frac\eps4 \int_{\T^2} \Dr_N^4 dx+\zeta \| \Dr_N \|^2_{H^1}+\int_{\T^2} \Dr_N^2 dx+O(\zeta^{-1})+O(1)
\label{CCZ110}
\end{align}

\noi
where the term $O(1)$ comes from Lemma \ref{LEM:Cor00} (i) by computing the expectation of  the higher moments for each component of $\Xi_N^{\eps}=\big(\eps^2\<4>_N, \eps^{\frac 32}\<3>_N, \eps \<2>_N, \sqrt{\eps} \<1>_N \big)$ in $C^{-\eta}$.  From Sobolev's and Young's inequalities, we have
\begin{align}
\E\Big[\| \sqrt{\eps} \<1>_N +\sqrt{\eps}\Dr_N  \|_{C^{-\eta}}^2 \Big]&\les 1+\sqrt{\eps} \E\Big[\| \Dr \|_{H^1_x}^2 \Big] \label{Idr0} \\ 
\bigg| \int_{\T^2} \Dr_N^2 dx \bigg| &\le \frac c{\eps}+\frac \eps{100}\|\Dr_N \|_{L^4}^4  \label{Idr}
\end{align}

\noi
for some $c \ge 1$. It follows from \eqref{logwtZ}, \eqref{UN0}, \eqref{CCZ110}, \eqref{Idr0}, \eqref{Idr}, and choosing $\zeta>0$ sufficiently small that 
\begin{align*}
-\log Z_{N,\eps, G} &\ge \frac K2+ \inf_{\dr \in \Ha}\E\bigg[  \Phi(\Xi_N^\eps, \sqrt{\eps} \Dr_N)-\kappa \| \sqrt{\eps} \<1>_N +\sqrt{\eps}\Dr_N  \|_{C^{-\eta}}^2 +\frac 14 \int_0^1 \| \dot{\Dr}(t) \|_{H^1_x}^2 dt     \bigg]\\
&\ge \frac K2-O(1)-\frac c{\eps} +\inf_{\dr \in \Ha}\frac{1}{10}\E\Big[\U(\Dr_N)\Big]   \\
&\ge \frac K2-O(1) -\frac{c}{\eps}
\end{align*}

\noi
uniformly in $N  \ge 1$. Hence, by choosing $M\gg 1$ (from \eqref{s2} and \eqref{G22}) and setting $K=\frac {4c}\eps \gg 1$, we obtain
\begin{align}
-\log Z_{\eps, G}=-\lim_{N\to \infty} \log Z_{N,\eps, G} \ge \frac K2-O(1)-\frac c\eps \ge \frac {c'}{2\eps}
\label{II}
\end{align}

\noi
for some $c'>0$.
Then, \eqref{wtZ} and \eqref{II} imply
\begin{align}
& \int_{  \{   \text{dist}(\sqrt{\eps}\phi, \M ) \ge \dl    \} \cap  \{\|\sqrt{\eps}\phi \|> M  \}    } F(\sqrt{\eps} \phi)\exp\bigg\{ -\frac{1}{\eps}\mathcal{V}(\sqrt{\eps}\phi) \bigg\} d\mu(\phi)   \notag \\
&=\lim_{N\to \infty} Z_{N,\eps, G} \les \exp\Big\{ -\frac {c'}{2\eps} \Big\}.
\label{IIa}
\end{align}

\noi
Combining \eqref{III}, \eqref{I0}, and \eqref{IIa} completes the proof of Proposition \ref{PROP:Var}.

\end{proof}

\begin{remark}\rm\label{REM:UBSY1}
Recall that $F$ is a continuous and unbounded observable with polynomial growth rates.
Since $F$ is a bounded functional on the region $\big\{\|\sqrt{\eps}\phi \|_{C^{-\eta}} \le M \big\}$, in order to obtain the estimate \eqref{I0} for the first part of the right-hand side of \eqref{III}, one can use the LDP for the $\Phi^4_2$-measure with Lemma \ref{LEM:sta}. We, however, point out that when considering the region  $\big\{\|\sqrt{\eps}\phi \|_{C^{-\eta}} > M \big\}$, it is not possible to use the LDP and so necessary to proceed with the argument above to obtain \eqref{IIa}.
\end{remark}

We conclude this subsection by presenting the proof of  Lemma \ref{LEM:Dr1}.
\begin{lemma} \label{LEM:Dr1}
\textup{(i)} Let $\eta>0$. For every $\zeta>0$, there exists $c(\zeta)>0$ such that
\begin{align}
\bigg| \int_{\T^2}  \<3>_{N}  \Dr_Ndx  \bigg|
&\le c(\zeta) \| \<3>_{N} \|_{C^{-\eta}}^2  
+ \zeta
\| \Dr_N\|_{ H^{1}}^2,   
\label{Y1}\\
\bigg| \int_{\T^2}  \<2>_{N} \Dr_N^2 dx \bigg|
&\le c(\zeta) \| \<2>_{N} \|_{C^{-\eta}}^4  
+ \zeta \Big(
\| \Dr_N\|_{ H^{1}}^2 + 
\| \Dr_N \|_{L^4}^4   \Big),   
\label{Y2} \\
\bigg| \int_{\T^2}  \<1>_{N} \Dr_N^3    dx\bigg|
&\le c(\zeta)  \| \<1>_{N}\|_{C^{-\eta}}^{c_0}+ \zeta \Big(
\| \Dr_N\|_{ H^{1}}^2 + 
\| \Dr_N \|_{L^4}^4   \Big), 
\label{Y3}
\end{align}

\noi
for some $c_0>0$ and every $N \in \N\cup \{\infty\} $.

\smallskip

\noi
\textup{(ii)} For every $\zeta>0$, there exists $c(\zeta)>0$ such that
\begin{align}
\bigg| \int_{\T^2}  \<2>_{N} \Dr_N dx \bigg|
&\le c(\zeta) \| \<2>_{N} \|_{C^{-\eta}}^2  
+ \zeta  \| \Dr_N \|_{H^1}^2, 
\label{YY22} \\
\bigg| \int_{\T^2}  \<1>_{N} \Dr_N^2  dx \bigg|
&\le c(\zeta) \| \<1>_{N} \|_{C^{-\eta}}^{c_1} + \zeta \Big(
\| \Dr_N \|_{H^1}^2 +  \| \Dr_N \|_{L^3}^3 \Big),
\label{YY3}
\end{align}

\noi
for some $c_1>0$ and every $N \in \N\cup \{\infty\}$.

\smallskip 

\noi
\textup{(iii)} For every $\zeta>0$, there exists $c(\zeta)>0$ such that
\begin{align}
\bigg| \int_{\T^2}  \<1>_{N} \Dr_N dx \bigg| \le c(\zeta) \|\<1>_N \|_{C^{-\eta}}^2 +\zeta \| \Dr_N \|_{H^1}^2
\label{YYY30}
\end{align}

\noi
for every $N \in \N \cup \{\infty\}$.

\end{lemma}

\begin{proof}

We first prove Part (i). From Young's inequality,  we have
\begin{align}
\begin{split}
\bigg| \int_{\T^2} \<3>_{N} \Dr_N dx \bigg|
&\le  \| \<3>_{N} \|_{C^{-\eta}} \| \Dr_N\|_{B^{\eta}_{1,1} }\le  \| \<3>_{N}\|_{C^{-\eta}} 
\| \Dr_N\|_{H^1}\\
&\le c(\zeta)\| \<3>_{N}\|_{C^{-\eta}}^2 + \zeta \| \Dr_N\|_{H^1}^2. 
\end{split}
\label{ZZZ3}
\end{align}
	
\noi
This yields  \eqref{Y1}. By the fractional Leibniz rule \eqref{fprod}, we have 
\begin{align}
\begin{split}
\bigg| \int_{\T^2}  \<2>_{N} \Dr_N^2 dx \bigg|
&  \le  \| \<2>_{N} \|_{C^{-\eta}}
\|\Dr_N^2\|_{B^{\eta}_{1,1} }\\
& \leq  \| \<2>_{N} \|_{C^{-\eta}} 
\|\Dr_N^2\|_{B^{\eta}_{\frac{4}{3},1 }}\\
& \les  \| \<2>_{N} \|_{C^{-\eta}}\|\Dr_N\|_{H^1}\|\Dr_N\|_{L^4}.
\end{split}
\label{ZZZ4}
\end{align}

\noi
Then, the second estimate \eqref{Y2}
follows from Young's inequality. Lastly, we consider \eqref{Y3}.
By the fractional Leibniz rule \eqref{fprod}, the interpolation inequality  \eqref{INT0P}, and Sobolev's inequality, we have
\begin{align}
\begin{split}
\bigg| \int_{\T^2}    \<1>_N   \Dr_N^3 dx \bigg|
&\le \| \<1>_N \|_{C^{-\eta}} \|   \Dr_N^3\|_{B_{1,1}^\eta  }\\
&\les\| \<1>_N \|_{C^{-\eta}} (\| \Dr_N  \|_{B_{2,1}^\eta} \| \Dr_N^2 \|_{L^2} +\| \Dr_N  \|_{L^2} \| \Dr_N^2 \|_{B_{2,1}^\eta} )   \\
&\les\| \<1>_N \|_{C^{-\eta}} (\| \Dr_N  \|_{B_{2,2}^{2\eta} } \| \Dr_N \|_{L^4}^2 +\| \Dr_N  \|_{L^2} \|\Dr_N \|_{L^4} \| \Dr_N \|_{B_{4,1}^\eta} )\\
&\les\| \<1>_N \|_{C^{-\eta}} 
\| \Dr_N \|_{H^1}^\be \| \Dr_N^0 \|_{L^4}^{3-\be}
\end{split}
\label{ZZZ2a}
\end{align}

\noi
for some $\be > 0$.
Then, the third  estimate \eqref{Y3}
follows from Young's inequality
since $\frac{\be}2 + \frac {3-\be}4 < 1$ for small $\be > 0$.

We now prove Part (ii). The estimate \eqref{YY22} follows
by replacing $\<3>_N$ in \eqref{ZZZ3}
by $\<2>_N$. It follows from 
the fractional Leibniz rule \eqref{fprod} and the interpolation inequality  \eqref{INT0P}
that
\begin{align*}
\bigg| \int_{\T^2}   \<1>_N \Dr_N^2 dx \bigg|
&  \le  \|  \<1>_N\|_{C^{-\eta}}
\|\Dr_N^2\|_{B_{1,1}^\eta}\\
& \leq
 \|  \<1>_N\|_{C^{-\eta}} 
\|\Dr_N\|^2_{H^{2\eta}}\\
& \les  \|  \<1>_N\|_{C^{-\eta}}
\|\Dr_N\|_{H^1}^{\be} \|\Dr_N\|_{L^2}^{2-\be}
\end{align*}

\noi
for some small $\be> 0$.
Then,  the second estimate \eqref{YY3} follows
from Young's inequality since
$\frac{\be} 2 + \frac{2-\be}3 < 1$.

Regarding Part (iii), the estimate \eqref{YYY30} follows
from replacing $\<3>_N$ in \eqref{ZZZ3} by $\<1>_N$.  This completes the proof of Lemma \ref{LEM:Dr1}.
\end{proof}

\section{Asymptotic expansion of the $\Phi^4_2$-measure}
\label{SEC:ASYM}
In this section, we obtain the asymptotic expansion presented in Theorem \ref{THM:1}.

\subsection{Expansion around a new Gaussian measure}
In this subsection, we present an expansion of $\Phi^4_2$-measure around a new Gaussian measure below (see \eqref{new}). Thanks to Proposition \ref{PROP:Var}, we have that for any $\dl>0$, there exists $c=c(\dl)>0$ such that
\begin{align*}
\int_{\{ \textup{dist}_{\mathcal{C}^{-\eta} }(\phi,\mathcal{M}) \ge \dl  \}   } F(\phi) \exp\bigg\{ -\frac 1\eps \mathcal{V}(\phi)\bigg\} d\mu_{\eps}(\phi)\les \exp\Big\{ -\frac c{2\eps} \Big\}.
\end{align*}

\noi
as $\eps \to 0$. In other words, Proposition \ref{PROP:Var} implies that the main contribution of $\int F(\phi) d\rho_\eps(\phi)$ comes from a neighborhood of the minimizers of the Hamiltonian \eqref{Ham0}. Hence, it suffices to consider the contribution on $\{\text{dist}_{C^{-\eta}}(\phi,\M) <\dl  \}$, that is,
\begin{align}
\int_{\{ \textup{dist}_{\mathcal{C}^{-\eta} }(\phi,\mathcal{M}) <\dl  \}   } F(\phi) \exp\bigg\{ -\frac 1\eps \mathcal{V}(\phi)\bigg\} d\mu_{\eps}(\phi).
\label{ins0}
\end{align}

\noi
The idea of performing the expansion is that
given $w\in \M$, we can write for small $v$ that
\begin{align}
H(w+v)=H(w)+\frac{1}{2}\jb{\nb^{(2)}H(w)v,v}+\text{error terms}
\label{expa001}
\end{align}

\noi
where $H(\phi)$ is as in \eqref{Ham0}.
Thus, the essential term in \eqref{expa001} is $\jb{\nb^{(2)}H(w)v,v}$. In particular, if $\nb^{(2)}H(w)>0$, then inserting \eqref{expa001} into 
\begin{align*}
\int_{ \{ \| v\|\ll 1 \} } F(w+v) \exp\Big\{ -\frac{1}{\eps} H(w+v) \Big\} dv
\end{align*}

\noi
shows that we can expand around a new Gaussian measure, with covariance operator $\big(\nb^{(2)}H(w)\big)^{-1}=(-\Dl+\nb^{(2)}V(w))^{-1}$, where $V(\phi)=\frac 14(|\phi|^2-1)^2$, which was conducted in Ellis and Rosen \cite{ER1} in the non-singular case. We, however, point out that because of the counter terms (i.e.~renormalizations) in the singular case, we cannot take the expansion directly as in Ellis and Rosen \cite{ER1}. Moreover, compared to the non-singular case, the Fredholm determinant associated with the partition function of the new Gaussian measure exhibits divergence. Consequently, we substitute it with the Carleman–Fredholm (renormalized) determinant, precisely compensating  the divergence of the original Fredholm determinant.

Before we move to the proof of the expansion, we first introduce notations which will be used later. For $w \in \M$, we define $\mu_w(dv)$ to be the Gaussian measure with the covariance operator
$(-\Dl+\nb^{(2)}V(w))^{-1}$. Then, we have
\begin{align}
c_{w,N}=\E_{\mu_w}\Big[|v_N(x)|^2\Big]=\sum_{|m|\le N}\frac{1}{|2\pi m|^2+\nb^{(2)}V(w)}=\sum_{|m|\le N}\frac{1}{|2\pi m|^2+2}\sim \log N
\label{cwN01}
\end{align}

\noi
as $N \to \infty$. We note that as $N$ approaches infinity, both $c_{N}$ in \eqref{tadpole1} and $c_{w,N}$ show a logarithmic divergence. Since the Gaussian reference measure is being changed, we define the Wick powers taken with respect to the new reference measure $\mu_w(dv)$ as follows: 
\begin{align*}
:\! v_N^2 \!:_{w}&=v_N^2-c_{w,N}\\
: \! v_N^3 \! :_{w}&=v_N^3-3c_{w,N}v_N\\
: \! v_N^4 \!:_{w}&=v_N^4-6c_{w,N}v_N^2+3c_{w,N}^2.
\end{align*}

\noi
We first consider \eqref{ins0} with the ultraviolet cutoff $\phi_N=\textbf{P}_N\phi$ and then take the limit $N \to \infty$ later in Lemma \ref{LEM:LNT} below.

\begin{lemma}\label{LEM:expN}
There exists $\dl_0>0$ such that for any  $F: \mathcal{C}^{-\eta}(\T^2;\R)\to \R$ with at most polynomial growth rate, any $\eps>0$, any $N \ge 1$, and every $0<\dl\le \dl_0$, we have
\begin{align}
&\int_{\{ \textup{dist}_{\mathcal{C}^{-\eta} }(\phi_N,\mathcal{M}) <\dl  \}   } F(\phi_N) \exp\bigg\{ -\frac 1\eps \mathcal{V}(\phi_N)\bigg\} \mu_{\eps}(d\phi) \notag \\
&=\sum_{w \in \M}\Bigg[\dr_{\textup{re},N}(w)  \int_{ \{ \| \sqrt{\eps}v_N \|_{C^{-\eta}} < \dl   \}  } F(w+\sqrt{\eps}v_N ) \notag \\
&\hphantom{XXXXXXXXX} \times \exp\Big\{-\frac \eps4 \H_4(v,c_{w,N},c_N) -\sqrt{\eps} \H_3(v,c_{w,N},c_N)\cdot w  \Big\}\mu_{w,N}(dv)  \Bigg],
\label{CHA01}
\end{align}

\noi
where $\mu_{w,N}:=(\textbf{P}_N)_\#\mu_w$ is the Gaussian measure with the covariance operator $\big(\P_N(-\Dl+\nb^{(2)}V(w) )\P_N\big)^{-1}$ for each $w\in \M$,
\begin{align}
\H_4(v_N,c_{w,N},c_N):&=\int_{\T^2}  :\! v_N^4 \!:_w  dx-6(c_N-c_{w,N}) \int_{\T^2} :\! v_N^2 \!:_w dx  \\
\H_3(v_N,c_{w,N},c_N):&=\int_{\T^2}  :\! v_N^3 \!:_w  dx- 3(c_N-c_{w,N}) \int_{\T^2} v_N dx 
\end{align}

\noi
and $\dr_{\textup{re},N }(w)$ is the renormalized Fredholm determinant 
\begin{align}
\exp\Big\{\frac{c_N}{2}\Big\}\bigg(\det\big(\Id_N+\P_N(1-\Dl)^{-1}(\nb^{(2)}V(w)-1)\P_N\big) \bigg)^{-\frac 12}\exp\Big\{-\frac{3\eps}{4}|c_{w,N}-c_N|^2 \Big\}.
\label{CF01}
\end{align}

\end{lemma}

\begin{proof}
By taking $\dl>0$ sufficiently small such that $B(1,\dl)\cap B(-1,\dl)=\emptyset $, where $B(w,R)$ denotes the ball of radius $R$ in $\mathcal{C}^{-\eta}(\T^2)$ centered at $w$, we have
\begin{align}
&\int_{\{ \textup{dist}_{\mathcal{C}^{-\eta} }(\phi_N,\mathcal{M}) <\dl  \}   } F(\phi_N) \exp\bigg\{ -\frac 1\eps \mathcal{V}(\phi_N)\bigg\} d\mu_{\eps}(\phi) \notag \\
&=\sum_{w \in \M}\int_{ \{ \| \phi_N -w \|_{C^{-\eta}} < \dl   \}  } F(\phi_N )\exp\Big\{-\frac 1\eps \mathcal{V}(\phi_N)     \Big\} d\mu_\eps(\phi) \notag \\
&= \sum_{w \in \M}\int_{ \{ \|  \sqrt{\eps}\phi_N -w \|_{C^{-\eta}} < \dl   \}  } F(\sqrt{\eps}\phi_N )\exp\Big\{-\frac 1\eps \mathcal{V}(\sqrt{\eps}\phi_N)     \Big\} d\mu(\phi).
\label{a0}
\end{align}

\noi
By using the translation $v\mapsto v+\sqrt{\eps}^{-1}w$  and denoting the translation map $T_{(\sqrt{\eps}^{-1}w)}(v)=v-\sqrt{\eps}^{-1}w$, we have
\begin{align}
&\sum_{w \in \M}\int_{ \{ \| \sqrt{\eps}\phi_N -w \|_{C^{-\eta}} < \dl   \}  } F(\sqrt{\eps} \phi_N )\exp\Big\{-\frac 1\eps \mathcal{V}(\sqrt{\eps} \phi_N)     \Big\} d\mu(\phi) \notag \\
&=\sum_{w \in \M}\int_{ \{ \| \sqrt{\eps}v_N \|_{C^{-\eta}} < \dl   \}  } F(w+\sqrt{\eps}v_N )\exp\Big\{-\frac 1\eps \mathcal{V}(w+\sqrt{\eps}v_N )     \Big\} d(T_{(\sqrt{\eps}^{-1}w)})_{\#} \mu(v)
\label{a1}
\end{align}

\noi
By recalling the Cameron–Martin formula for the free field $\mu$ with covariance operator $(1-\Dl)^{-1}$, we have
\begin{align}
\frac{d(T_{h})_{\#} \mu}{d\mu}(v)=\exp\Big\{-\frac 12 \big\langle(1-\Dl)h,h \big \rangle_{L^2}  -\big\langle(1-\Dl)h, v\big\rangle_{L^2} \Big\}
\label{a2}
\end{align}

\noi
where $h$ is an element of the Cameron–Martin space $H^1(\T^2)$. By combining \eqref{a0}, \eqref{a1}, and \eqref{a2}, we have
\begin{align}
&\int_{\{ \textup{dist}_{\mathcal{C}^{-\eta} }(\phi_N,\mathcal{M}) <\dl  \}   } F(\phi_N) \exp\bigg\{ -\frac 1\eps \mathcal{V}(\phi_N)\bigg\} d\mu_{\eps}(\phi) \notag \\
&=\sum_{w \in \M}\int_{ \{ \| \sqrt{\eps}v_N \|_{C^{-\eta}} < \dl   \}  } F(w+\sqrt{\eps}v_N )\exp\Big\{-\frac 1\eps \mathcal{V}(w+\sqrt{\eps}v_N ) -\frac{1}{2\eps}-\frac{w}{\sqrt{\eps}}\int_{\T^2} v_N  dx    \Big\}   d\mu(v)
\label{CM01}
\end{align}

\noi
where we used the fact that $w \in \M=\{-1,1\}$ is a constant. We recall the identity
\begin{align}
H_k(x+y; \s )&  = 
\sum_{\l = 0}^k
\begin{pmatrix}
k \\ \l
\end{pmatrix}
 x^{k - \l} H_\l(y; \s),
\label{Herm}
\end{align}

\noi
where $H_k(x;\s)$ is the Hermite polynomial of degree $k$. Therefore, from \eqref{V12} and \eqref{Herm}, we have 
\begin{align}
\mathcal{V}(w + \sqrt{\eps} v_N ) & =\mathcal{V}_{1}(w + \sqrt{\eps} v_N)+\mathcal{V}_{2}(w + \sqrt{\eps} v_N) \notag \\
&=\mathcal{V}_{1,1}(w + \sqrt{\eps} v_N)-\mathcal{V}_{1,2}(w + \sqrt{\eps} v_N)+\frac 14+\mathcal{V}_2(w + \sqrt{\eps} v_N )
\label{a20}
\end{align} 

\noi
where
\begin{align}
\mathcal{V}_{1,1}(w + \sqrt{\eps} v_N)&=\frac 14\int_{\T^2}  :\! (\sqrt{\eps}v_N)^4 \!:  dx
+ \int_{\T^2}  :\! ((\sqrt{\eps}v_N))^3 \!: w dx+\frac {3}2\int_{\T^2}  :\! (\sqrt{\eps}v_N)^2 \!:  w^2 dx \notag \\
&\hphantom{X}
+ \int_{\T^2} (\sqrt{\eps}v_N)  w^3 dx
+ \frac 14\int_{\T^2} w^4 dx,   \label{a21}\\
\mathcal{V}_{1,2}(w + \sqrt{\eps} v_N)&=\frac 12 \int_{\T^2}  :\! (\sqrt{\eps}v_N)^2 \!:  dx+\int_{\T^2}   (\sqrt{\eps}v_N) w  dx+ \frac 12 \int_{\T^2}  w^2 dx   \label{a3}\\
\mathcal{V}_{2}(w + \sqrt{\eps} v_N)&=-\frac 12 \int_{\T^2}  :\! (\sqrt{\eps}v_N)^2 \!:  dx-\int_{\T^2}   (\sqrt{\eps}v_N) w  dx- \frac 12 \int_{\T^2}  w^2 dx.
\label{a4}
\end{align}

\noi
It follows from \eqref{CM01}, \eqref{a20}, \eqref{a21}, \eqref{a3} and \eqref{a4} that\footnote{Recall that we endow $\T^2$ with the normalized Lebesgue measure $dx_{\T^2}= (2\pi)^{-2} dx$.} 
\begin{align}
&\frac 1\eps\mathcal{V}(w + \sqrt{\eps} v_N)+\frac{1}{2\eps}+\frac{w}{\sqrt{\eps}} \int_{\T^2} v_N dx \notag \\
&= \frac \eps4\int_{\T^2}  :\! v_N^4 \!:  dx +\sqrt{\eps} \int_{\T^2}  :\! v_N^3 \!: w dx+\frac{1}{2}\int_{\T^2}  v_N^2 dx-\frac{c_N}{2}
\label{WICKO}
\end{align}

\noi
where $c_N$ is as in \eqref{tadpole1}, which is a divergent constant as $N\to\infty$.
By combining \eqref{WICKO} with \eqref{CM01}, we have
\begin{align}
&\int_{\{ \textup{dist}_{\mathcal{C}^{-\eta} }(\phi_N,\mathcal{M}) <\dl  \}   } F(\phi_N) \exp\bigg\{ -\frac 1\eps \mathcal{V}(\phi_N)\bigg\} d\mu_{\eps}(\phi) \notag \\
&=\sum_{w \in \M}\int_{ \{ \| \sqrt{\eps}v_N \|_{C^{-\eta}} < \dl   \}  } F(w+\sqrt{\eps}v_N ) \notag \\
&\hphantom{XXXXXXXXXX}\times \exp\Big\{-\frac \eps4\int_{\T^2}  :\! v_N^4 \!:  dx -\sqrt{\eps} \int_{\T^2}  :\! v_N^3 \!: w dx-\frac{1}{2}\int_{\T^2}  v_N^2 dx+\frac{c_N}{2}\Big\}   d\mu(v).
\label{CCM01}
\end{align}

\noi
Regarding the term $-\frac{1}{2}\int_{\T^2} v_N^2 dx$ on the right hand side of \eqref{CCM01}, we notice that\footnote{Recall that $\nb^{(2)}V(u)=3|u|^2-1$ and so $\nb^{(2)}V(w)=2$ for any $w \in \M$.}  
\begin{align}
\exp\Big\{- \frac{1}{2} \int_{\T^2} |v_N|^2 dx \Big\}=\exp\Big\{-\frac 12 \big\langle (\nb^{(2)} V(w)-1)v_N, v_N \big\rangle_{L^2} \Big\} 
\label{S1}
\end{align}

\noi
for any $w\in \M=\{-1,1\}$, where $V(u)=\frac{1}{4}\big(|u|^2-1\big)^2$. Hence, by combining \eqref{S1} with the Gaussian free field $\mu$, we introduce a new Gaussian measure with the covariance operator 
\begin{align*}
B_{w,N}:=(\P_N(1-\Dl)\P_N+\P_N(\nb^{(2)}V(w)-1) \P_N)^{-1}=(\P_N(-\Dl+\nb^{(2)}V(w)) \P_N)^{-1}
\end{align*} 

\noi
as follows
\begin{align}
\mu_{w,N}(dv)&=\sqrt{\det \big(\Ld_N\big)}\exp\Big\{-\frac 12 \big\langle (\nb^{(2)}V(w)-1)v_N,v_N\big\rangle_{L^2} \Big\}d\mu_N(v) \notag \\
&=\frac{1}{\sqrt{\det\big(B_{w,N}\big)}}\exp\Big\{-\frac 12 \big\langle (-\Dl+\nb^{(2)}V(w))v_N,v_N\big\rangle_{L^2} \Big\} dv_N
\label{S2}
\end{align}

\noi
where $\mu_N$ is the Gaussian measure with covariance $A_N=\P_N(1-\Dl)^{-1}\P_N$ and
\begin{align*}
B_{w,N}&=\sqrt{A_{N}} \big(\Id_N+\sqrt{A_{N}}(\nb^{(2)}V(w)-1)\sqrt{A_N} \big)^{-1}     \sqrt{A_{N}}\\
&= \sqrt{A_{N}}  \Ld_N^{-1} \sqrt{A_{N}}
\end{align*}


\noi
where $\Ld_N:=\Id_N+\sqrt{A_{N}}(\nb^{(2)}V(w)-1)\sqrt{A_N}$.
We also note that $\det\big(\Ld_N\big)$ can be written as
\begin{align*}
\det\big(\Ld_N\big)&=\det\big(\Id_N+\P_N(1-\Dl)^{-\frac 12}(\nb^{(2)}V(w)-1)(1-\Dl)^{-\frac 12}\P_N\big)\\
&=\det\big(\Id_N+\P_N(1-\Dl)^{-1}(\nb^{(2)}V(w)-1)\P_N\big).
\end{align*}


\noi
Notice that the new Gaussian measure $\mu_{w,N}$ converges weakly to the limiting Gaussian measure 
\begin{align}
\mu_{w}(dv)=\frac{1}{\sqrt{\det\big(B_{w}\big)}}\exp\Big\{-\frac 12 \big\langle (-\Dl+\nb^{(2)}V(w))v,v\big\rangle_{L^2} \Big\}dv
\label{new}
\end{align} 

\noi
with the covariance operator $C_w=(-\Dl+\nb^{(2)}V(w) )^{-1}$ for any $w\in \M$. Hence, we need to rewrite the Wick powers $:v_N^4:$ and $:v_N^3:$ in \eqref{WICKO} in terms of the new Gaussian measure $\mu_w$ in \eqref{new}. We note that
for any Gaussian $X$ and $\s_1, \s_2$ in $\R$, we have 
\begin{align}
H_1(X;\s_1)&=H_1(X;\s_2) \label{Her1} \\
H_2(X;\s_1)&=H_2(X;\s_2)-(\s_1-\s_2) \label{Her2} \\
H_3(X;\s_1)&=H_3(X;\s_2)-3(\s_1-\s_2)H_1(X,\s_2) \label{Her3} \\
H_4(X;\s_1)&=H_3(X;\s_2)-6(\s_1-\s_2)H_2(X,\s_2)+3(\s_1-\s_2)^2 \label{Her4}
\end{align}

\noi
where $H_k(x;\s)$ is the Hermite polynomial of degree $k$. Hence, it follows from \eqref{Her1}, \eqref{Her2}, \eqref{Her3}, and \eqref{Her4} that
\begin{align}
:\!v_N^2\!:&=:\!v_N^2\!:_{w}-(c_N-c_{w,N}), \label{S20}\\
:\!v_N^3\!:&=:\!v_N^3\!:_{w}-3(c_N-c_{w,N})v_N, \label{S3}\\
:\!v_N^4\!:&=:\!v_N^4\!:_{w}-6(c_N-c_{w,N}):\!v_N^2\!:_w+3|c_{N}-c_{w,N}|^2
\label{S4}
\end{align}

\noi
where 
\begin{align*}
:\! v_N^2 \!:_{w}&=v_N^2-c_{w,N}\\
: \! v_N^3 \! :_{w}&=v_N^3-3c_{w,N}v_N\\
: \! v_N^4 \!:_{w}&=v_N^4-6c_{w,N}v_N^2+3c_{w,N}^2
\end{align*}

\noi
and
\begin{align}
c_{w,N}=\E_{\mu_w}\Big[|v_N(x)|^2\Big]=\sum_{|m|\le N}\frac{1}{|2\pi m|^2+\nb^{(2)}V(w)}=\sum_{|m|\le N}\frac{1}{|2\pi m|^2+2}\sim \log N
\label{cwNdiv}
\end{align}

\noi
as $N \to \infty$. We note that as $N$ approaches infinity, both $c_{N}$ in \eqref{tadpole1} and $c_{w,N}$ in \eqref{cwNdiv} exhibit logarithmic divergence. However, the difference between them, $c_{N} - c_{w,N}$, converges as $N$ tends to infinity since
\begin{align}
c_{N}-c_{w,N}&=\sum_{|m|\le N}\bigg(\frac{1}{|2\pi m|^2+1}-\frac{1}{|2\pi m|^2+2}\bigg) \notag \\
&=\sum_{|m|\le N} \frac{1}{(|2\pi m|^2+2)(|2\pi m|^2+1)}<\infty,
\label{diffcov}
\end{align}

\noi
uniformly in $N \ge 1$. Hence, by using  \eqref{S20}, \eqref{S3}, and \eqref{S4}, we rewrite the terms in \eqref{CCM01} as follows
\begin{align}
\int_{\T^2}  :\! v_N^4 \!:  dx&=\int_{\T^2}  :\! v_N^4 \!:_w  dx-6(c_N-c_{w,N}) \int_{\T^2} :\! v_N^2 \!:_w dx++3|c_{N}-c_{w,N}|^2 \label{JJJM1}\\
\int_{\T^2}  :\! v_N^3 \!:  dx&=\int_{\T^2}  :\! v_N^3 \!:_w  dx- 3(c_N-c_{w,N}) \int_{\T^2} v_N dx \label{JJJM2}. 
\end{align}

\noi
It follows from \eqref{CCM01}, \eqref{S1}, \eqref{S2}, \eqref{JJJM1}, and \eqref{JJJM2} that 
\begin{align}
&\int_{\{ \textup{dist}_{\mathcal{C}^{-\eta} }(\phi_N,\mathcal{M}) <\dl  \}   } F(\phi_N) \exp\bigg\{ -\frac 1\eps \mathcal{V}(\phi_N)\bigg\} d\mu_{\eps}(\phi) \notag \\
&=\sum_{w \in \M}\Bigg[\dr_{\textup{re},N}(w)  \int_{ \{ \| \sqrt{\eps}v_N \|_{C^{-\eta}} < \dl   \}  } F(w+\sqrt{\eps}v ) \notag \\
&\hphantom{XXXXXXXXX} \times \exp\Big\{-\frac \eps4 \H_4(v,c_{w,N},c_N) -\sqrt{\eps} \H_3(v,c_{w,N},c_N)\cdot w  \Big\}d\mu_{w,N}(v)  \Bigg]
\label{IIN1}
\end{align}

\noi
where 
\begin{align}
\H_4(v_N,c_{w,N},c_N):&=\int_{\T^2}  :\! v_N^4 \!:_w  dx-6(c_N-c_{w,N}) \int_{\T^2} :\! v_N^2 \!:_w dx \label{H4}\\
\H_3(v_N,c_{w,N},c_N):&=\int_{\T^2}  :\! v_N^3 \!:_w  dx- 3(c_N-c_{w,N}) \int_{\T^2} v_N dx \label{H3}
\end{align}

\noi
and $\dr_{\textup{re},N }(w)$ is the renormalized Fredholm determinant
\begin{align}
&\exp\Big\{\frac{c_N}{2}\Big\} \Big(\det\big(\Ld_N\big)\Big)^{-\frac{1}{2}}\exp\Big\{-\frac{3\eps}{4}|c_N-c_{w,N}|^2 \Big\}.
\label{Redet}
\end{align}

\noi
This completes the proof of Lemma \ref{LEM:expN}.
\end{proof}

First recall that for a linear operator $L$ on a Hilbert space $H$ with eigenvalues $\{\ld_k\}_{k\in J}$, where $J$ is a
countable set, the Fredholm determinant $\det(\Id+L)$ is defined by
\begin{align*}
\det(\Id+L)=\prod_{k \in J}(1+\ld_k).
\end{align*}

\noi
Notice that the Fredholm determinant $\det\big(\Id_N+\P_N(1-\Dl)^{-1}(\nb^{(2)}V(w)-1)\P_N\big)$ in \eqref{Redet} appears as a consequence of transforming the Gaussian reference measure. We, however, note that compared to previous cases \cite{ER1, GT}, the Fredholm determinant is not well-defined, since $(1-\Dl)^{-1}$ is not of trace class in $L^2(\T^2)$, where $\text{Tr}\big[(-\Delta+1)^{-1} \big]$ exhibits a logarithmic divergence. It is important to observe that  the additional term $\exp\big\{\frac{c_N}{2}\big\}$ in \eqref{CF01} precisely compensates the divergence of the Fredholm determinant, replacing it by a so-called Carleman–Fredholm renormalized determinant. 
More precisely, we introduce the modified determinant 
\begin{align*}
\det(\Id + L )e^{- \text{Tr} L}
\end{align*}

\noi
of a linear operator $L$, as appearing in \eqref{Redet}, commonly known as the Carleman–Fredholm determinant, which is defined for every Hilbert–Schmidt perturbation $L$ of the identity, without the requirement of $L$ to be of trace-class, see \cite[Chapter 5]{BSim}. The same phenomenon was observed in \cite[Section 2]{BDW} in the context of obtaining the Eyring--Kramers formula.

In the following lemma, we show that thanks to the renormalization factor effect $\exp\big\{\frac{c_N}{2}\big\}$, the renormalized determinant $\dr_{\textup{re},N }$ in \eqref{Redet} converges as $N \to \infty$.

\begin{lemma}\label{LEM:redet}
For every $w\in \M$, the  renormalized determinant $\dr_{\textup{re},N }(w)$ defined in \eqref{CF01} converges as $N\to \infty$ where $c_N$ and $c_{w,N}$ are given by \eqref{tadpole1} and \eqref{cwN01}, respectively.
\end{lemma}

\begin{proof}
Both $c_{N}$ in \eqref{tadpole1} and $c_{w,N}$ in \eqref{cwNdiv} show logarithmic divergence. However, their difference $c_{N} - c_{w,N}$ converges as $N$ tends to infinity thanks to \eqref{diffcov}. Hence, it suffices to consider the first two terms in \eqref{CF01}.

Note that when $m \in \Z^2$, we have
\begin{align*}
\big(\P_N(-\Dl+1)^{-1}(\nb^{(2)}V(w)-1)\P_N\big)e^{2\pi i m\cdot x}=(1+|2\pi m|^2)^{-1} e^{2\pi i m\cdot x}.
\end{align*}


\noi
Hence, we obtain
\begin{align}
\det\big(\Id_N+\P_N(1-\Dl)^{-1}(\nb^{(2)}V(w)-1)\P_N\big)&=\prod_{|m|\le N} \Big(1+\frac{1}{1+|2\pi m|^2} \Big) \notag \\
&=\bigg(\prod_{ |m| \le N} \frac{2+|2\pi m|^2}{1+|2\pi m|^2} \bigg). 
\label{e3}
\end{align}

\noi
Recall that when $a_m>0$, $\prod\limits_{m}(1+a_m)$ converges if and only if $\sum\limits_{m}a_m <\infty$. Hence, \eqref{e3} diverges as $N\to \infty$ since $\sum\limits_{m\in \Z^2}(1+|2\pi m|^2)^{-1}=\infty$. We now exploit the effect $\exp\Big\{ \frac{c_N}{2} \Big\}$ to ensure the convergence of the determinant. It follows from \eqref{tadpole1} and \eqref{e3} that
\begin{align}
&\exp\Big\{\frac{c_N}{2}\Big\}\bigg(\det\big(\Id_N+\P_N(1-\Dl)^{-1}(\nb^{(2)}V(w)-1)\P_N\big) \bigg)^{-\frac 12} \notag \\
&=\Bigg(\prod_{|m|\le N} \frac{ (1+|2\pi m|^2) \exp\big\{\frac{1}{1+|2\pi m|^2} \big\}   }{2+|2\pi m|^2 } \Bigg)^{\frac 12}=\Bigg(\prod_{|m|\le N} \frac{e^{1/\ld_m }}{1+1/\ld_m} \Bigg)^{\frac 12},
\label{proexp}
\end{align}

\noi
where $\ld_m=1+|2\pi m|^2$. Thanks to 
\begin{align*}
\log\bigg(\frac{e^{x}}{1+x}\bigg)=x-\log(1+x) \le \frac{1}{2}x^2,
\end{align*}

\noi
for $x>0$, we have that
\begin{align}
\sum_{|m|\le N} \log\bigg(  \frac{e^{1/\ld_m }}{1+1/\ld_m} \bigg)\le \frac 12\sum_{|m|\le N} \ld_m^{-2} <\infty,
\label{sumlog}
\end{align}

\noi
uniformly in $N \ge 1$. We emphasize that, due to the renormalization effect $e^{1/\ld_m}$, it was possible to obtain the sum of $\ld_m^{-2}$ over $m\in \Z^2$ (instead of the sum of $\ld_m^{-1}$). Therefore, \eqref{sumlog} shows that \eqref{proexp} converges as $N\to \infty$, from which we obtain the result.

\end{proof}


Thanks to Lemma \ref{LEM:redet} and \eqref{diffcov}, we have that for any $w\in \M$
\begin{align}
\dr_{\text{Re,N}}(w) &\to \dr_{\text{Re}}(w) \label{REDET0}\\ 
\H_3(v_N,c_{w,N},c_N) & \to \H_3(v,c_{w},c) \label{H3c}\\
\H_4(v_N,c_{w,N},c_N) & \to \H_4(v,c_{w},c) \label{H4c}
\end{align}

\noi
in $L^p(d\mu_w)$ for any $1\le p<\infty$ as $N \to \infty$. For the convergence of \eqref{H3c} and \eqref{H4c}, see, for example, \cite[Proposition 1.1]{OTh}. In order to take the limit $N\to \infty$ in \eqref{CHA01}, we need the following uniform exponential integrability. 

\begin{lemma}\label{LEM:UNIF}
Let $\dl>0$ and $1\le p<\infty$. Then, there exists a constant $C_{\dl,p}>0$ and $\eps_0>0$ such that for any $0< \eps \le \eps_0$ and any $w\in \M$, we have
\begin{align}
\int_{ \{ \| \sqrt{\eps}v \|_{C^{-\eta}} < \dl   \}  }  \exp\Big\{-\frac {\eps p}{4} \H_4(v_N,c_{w,N},c_N) -\sqrt{\eps}p \H_3(v_N,c_{w,N},c_N)\cdot w  \Big\}\mu_{w}(dv)\le C_{\dl, p}<\infty,
\label{UNIFORM}
\end{align}

\noi
uniformly in $N\ge 1$, where $\H_4$ and $\H_3$ are as in \eqref{H4} and \eqref{H3} and
$\mu_w$ is the Gaussian measure with the covariance $(-\Dl+\nb^{(2)}V(w) )^{-1}$.

\end{lemma}

\begin{proof}[Proof of Lemma \ref{LEM:UNIF}]
Note that
\begin{align}
 \ind_{\{|\,\cdot \,| \le K\}}(x) \le \exp\big( -  A |x|^\gamma\big) \exp(A K^\g)
\label{H6}
\end{align}

\noi
for any $K, A , \g > 0$. Hence, in order to prove \eqref{UNIFORM}, it suffices to prove that there exists $\eps_0>0$ such that for any $0< \eps \le \eps_0$, 
\begin{align}
\wt Z_{N,A}&=\E_{\mu_{w} }\bigg[  \exp\Big\{-\frac {\eps p }{4} \H_4(v_N,c_{w,N},c_N) -\sqrt{\eps }p \H_3(v_N,c_{w,N},c_N)\cdot w   -A\| \sqrt{\eps} v_N   \|_{C^{-\eta}}^2 \Big\}   \bigg] \notag \\
&=\E_{\PP }\bigg[  \exp\Big\{-\frac {\eps p }{4} \H_4(\<1>_{w,N},c_{w,N},c_N) -\sqrt{\eps }p \H_3(\<1>_{w,N},c_{w,N},c_N)\cdot w   -A\| \sqrt{\eps} \<1>_{w,N}   \|_{C^{-\eta}}^2 \Big\}   \bigg] \notag \\
&\le C_{\dl,p}<\infty,
\label{UNE1}
\end{align}

\noi
uniformly in $N\ge 1$, where $A=A(p)$ will be specified later. By the Bou\'e-Dupuis variational formula (Lemma \ref{LEM:var3}), we have
\begin{align}
-\log \wt Z_{N,A}&=\inf_{\dr \in \Ha} \E_\PP\bigg[   \frac{\eps p}{4}\H_4(\<1>_{w,N}+\Dr_N,c_{N,w}, c_N)+ \sqrt{\eps}p\H_3(\<1>_{w,N}+\Dr_N,c_{w,N}, c_N) \notag\\
&\hphantom{XXXXXX}+A\|\<1>_{w,N}+\Dr_N \|_{C^{-\eta}}^2+\frac{1}{2} \int_{0}^1 \|\dr(t) \|_{L^2}^2 dt \bigg] \notag \\
&=\inf_{\dr \in \Ha} \E_{\PP} \big[ \W_N(\dr)\big]
\label{UNE2}
\end{align}

\noi
where $\Dr(t)$ is defined by 
\begin{align}
\Dr(t) = \int_0^t \big(-\Dl+\nb^{(2)}V(w)\big)^{-\frac 12} \dr(t') dt'.
\end{align}

\noi
In view of the Bou\'e-Dupuis formula (Lemma \ref{LEM:var3}), 
it suffices to  establish a  lower bound on $\W_N(\dr)$ in \eqref{UNE2}, uniformly in $N \in \N$ and  $\dr \in \Ha$.
From \eqref{H4}, \eqref{H3}, and the identity
\begin{align*}
H_k(x+y; \s )&  = 
\sum_{\l = 0}^k
\begin{pmatrix}
k \\ \l
\end{pmatrix}
 x^{k - \l} H_\l(y; \s),
\end{align*}

\noi
where $H_k(x;\s)$ is the Hermite polynomial of degree $k$, we have
\begin{align}
\begin{split}
\frac{\eps p}{4}\H_4 (\<1>_{w,N} + \Dr_N,c_{w,N},c_N)  & = 
\frac {\eps p}4\int_{\T^2}  \<4>_{w,N}  dx
+ \frac {3\eps p}4 \int_{\T^2}  \<3>_{w,N} \Dr_N dx+\frac {3\eps p}2\int_{\T^2}  \<2>_{w,N}  \Dr_N^2 dx
\\
&\hphantom{X}
+ \frac{3\eps p}{4} \int_{\T^2} \<1>_{w,N}  \Dr_N^3 dx
+ \frac {\eps p}4 \int_{\T^2} \Dr_N^4 dx, \\
&\hphantom{X}
-\frac{3\eps p}{2}(c_{N}-c_{w,N}) \int_{\T^2} \<2>_{w,N} dx-3\eps p(c_{N}-c_{w,N}) \int_{\T^2} \<1>_{w,N} \Dr_N dx\\
&\hphantom{X}
-\frac{3\eps p}{2}(c_{N}-c_{w,N})\int_{\T^2}  \Dr_N^2 dx
\end{split}
\label{H4a}
\end{align}	

\noi
and
\begin{align}
\sqrt{\eps}p\H_3(\<1>_{w,N}+ \Dr_N, c_{w,N}, c_N)  & = 
\sqrt{\eps}p\int_{\T^2}  \<3>_{w,N}  dx
+ 3\sqrt{\eps}p \int_{\T^2}  \<2>_{w,N}  \Dr_N dx+3\sqrt{\eps}p\int_{\T^2} \<1>_{w,N}  \Dr_N^2 dx \notag \\
&\hphantom{X}
-3\sqrt{\eps}p(c_{N}-c_{w,N})\int_{\T^3}(\<1>_{w,N}+\Dr_N )dx + \sqrt{\eps}p \int_{\T^2} \Dr_N^3 dx.
\label{H3b}
\end{align}

\noi
Moreover, since 
\begin{align*}
(a+b)^2
\ge \frac 12 b^2 - 2a^2
\end{align*}
	
\noi
for any $a,b,c \in \R$, we have 
\begin{align}
A\| \<1>_{w,N}+\Dr_N \|_{C^{-\eta}}^2 \ge \frac{A}{2} \| \Dr_N \|^2_{C^{-\eta} }-2 \|\<1>_{w,N} \|_{C^{-\eta}}^2.
\label{ANCOER}
\end{align}

\noi
The main strategy now is to establish a pathwise lower bound on $\W_N(\dr)$ in~\eqref{UNE2}, 
uniformly in $N \in \N$ and $\dr \in \Ha$, by making use of the positive terms:
\begin{equation}
\E \bigg[\frac {\eps p} 4\| \Dr_N\|_{L^4}^4 + \frac A2\| \Dr_N \|^2_{H^{-\eta}}    + \frac{1}{2} \int_0^1 \| \dr(t) \|_{L^2_x}^2 dt\bigg]
\label{v_N1}
\end{equation}

\noi
coming from \eqref{H4a} and \eqref{H3b}. Then, by applying Lemmas \ref{LEM:Cor00} and \ref{LEM:Dr1} to \eqref{UNE2} together with \eqref{H4a} and \eqref{H3b} and choosing $\eps>0 $ sufficiently small, we have
\begin{align*}
&\E\Bigg[ \bigg|\frac{\eps p}{4}\H_4 (\<1>_{w,N} + \Dr_N,c_{w,N},c_N)-\frac{\eps p}{4} \int_{\T^2} \Dr_N^4 dx\bigg| \Bigg]\\
&\le C_0+\zeta' \frac{\eps p}{4} \int_{\T^2} \Dr_N^4 dx+\zeta' \|(-\Dl+\nb^{(2)}V(w) )^{\frac 12}\Dr_N \|_{L^2}^2 
\end{align*}

\noi
and
\begin{align*}
&\E\Bigg[\bigg|\sqrt{\eps}p\H_3(\<1>_{w,N}+ \Dr_N, c_{w,N}, c_N)-p\sqrt{\eps} \int_{\T^2}\Dr_N^3  wdx\bigg| \Bigg]\\
& \le C_0+\zeta' p \sqrt{\eps} \int_{\T^2} |\Dr_N|^3 dx +\zeta' \|(-\Dl+\nb^{(2)}V(w) )^{\frac 12}\Dr_N \|_{L^2}^2
\end{align*}

\noi
for arbitrary small $\zeta'>0$, which implies
\begin{align}
\W_N(\dr) &\ge -C_0+\E\bigg[\frac{\eps p}{4+2\zeta} \int_{\T^2} \Dr_N^4 dx+ 
\frac A2\| \Dr_N \|^2_{C^{-\eta}} +\frac{1}{2+\zeta}\| (-\Dl+\nb^{(2)}V(w) )^{\frac 12}  \Dr_N \|_{L^2}^2 \notag \\
&\hphantom{XXXXXXX}+p\sqrt{\eps} \int_{\T^2}\Dr_N^3  wdx-\zeta p \sqrt{\eps} \int_{\T^2} |\Dr_N|^3 dx \bigg]
\label{J11}
\end{align}

\noi
for arbitrary small $0<\zeta \ll 1$, where $C_0$ arises from Lemma \ref{LEM:Cor00} (i) by computing the expected values of the higher moments for each component of $\Xi_{w, N}=\big(\<4>_{w,N}, \<3>_{w,N}, \<2>_{w,N},  \<1>_{w,N} \big)$ in $C^{-\eta}$, uniformly in $N\ge 1$. We may assume that $w=1$. On the set $\{ \Dr_N \ge 0\}$, we have $\W_N(\dr) \ge -C_0$ uniformly in $N \in \N$ and $\dr \in \Ha$. Hence, it is enough to consider the case $\{\Dr_N <0 \}$. Note that 
\begin{align}
&\frac{\eps p}{4+2\zeta} \int_{\{ \Dr_N < 0\}} \Dr_N^4 dx+p\sqrt{\eps} \int_{\{ \Dr_N < 0\}}\Dr_N^3  wdx-\zeta p \sqrt{\eps} \int_{\{ \Dr_N < 0\}} |\Dr_N|^3 dx \notag \\
&=\frac{\eps p}{4+2\zeta} \int_{ \{ \Dr_N < 0\}   } \Dr_N^4 dx+ p\sqrt{\eps}(1+\zeta) \int_{ \{\Dr_N <0  \}  }\Dr_N^3 dx \notag \\
&=\frac{1}{4+2\zeta}\int_{\{ \Dr_N < 0\} } \Dr_N^2 \Big(  \eps p \Dr_N^2 +(4+2\zeta)p \sqrt{\eps} (1+\zeta) \Dr_N    \Big) dx \notag \\
&= \frac{1}{4+2\zeta}\int_{\{ \Dr_N < 0\} } \Dr_N^2 \bigg\{ \Big(  \sqrt{\eps p} \Dr_N+(2+\zeta) \sqrt{p} (1+\zeta)   \Big)^2  -p(2+\zeta)^2(1+\zeta)^2     \bigg\} dx \notag \\
& \ge -\frac{p(2+\zeta)^2(1+\zeta)^2}{4+2\zeta} \int_{\T^2} \Dr_N^2 dx.
\label{J12}
\end{align}

\noi
Since $\zeta$ is arbitrary small, it suffices to consider
\begin{align}
\frac{A}{2} \| \Dr_N \|_{C^{-\eta}}^2 +\frac{1}{2}\| (-\Dl+\nb^{(2)}V(w) )^{\frac 12}  \Dr_N \|_{L^2}^2 -p \|  \Dr_N\|_{L^2}^2
\label{J13}
\end{align}

\noi
where the first two terms come from \eqref{J11} and the last term arises from \eqref{J12}.
Notice that $\| u\|_{C^{-\eta}} \ges \| u\|_{H^{-2\eta}}$ and so
\eqref{J13} can be written as
\begin{align}
\frac 12 \sum_{n\in \Z^2} \big(A\jb{n}^{-4\eta}+ |n|^2+2-2p \big)|\ft \Dr_N(n) |^2.
\label{J14}
\end{align}

\noi
The worst case happens at the zero frequency and so in order for \eqref{J14} to be nonnegative, $A+2-2p>0$ if and only if $p<\frac{A}{2}+1$. It follows from \eqref{J11}, \eqref{J12}, \eqref{J13}, \eqref{J14}, and choosing $A=A(p)$ sufficiently large that 
\begin{align}
\inf_{N \in \mathbb{N}} \inf_{\dr \in \Ha} \W_N(\dr) 
\geq 
\inf_{N \in \mathbb{N}} \inf_{\dr \in \Ha}
\Big\{ -C_0 +\eqref{J14} \Big\}
\geq - C_0 >-\infty.
\label{UNE3}
\end{align}

\noi
Therefore, the uniform exponential integrability \eqref{UNE1} follows from \eqref{UNE2} and \eqref{UNE3}.

\end{proof}

\begin{remark}\rm\label{REM:UNICON}
The uniform $L^p$-bound in \eqref{UNIFORM} together with the
convergence as a consequence of \eqref{H3c} and \eqref{H4c} implies the following $L^p$-convergence of the density:
\begin{align*}
&\exp\Big\{-\frac {\eps }{4} \H_4(v_N,c_{w,N},c_N) -\sqrt{\eps} \H_3(v_N,c_{w,N},c_N)\cdot w  \Big\} \ind_{ \{ \|\sqrt{\eps}v \|_{C^{-\eta}} <\dl \} }\\
&\too \exp\Big\{-\frac {\eps }{4} \H_4(v,c_{w},c) -\sqrt{\eps} \H_3(v,c_{w},c)\cdot w  \Big\} \ind_{ \{ \|\sqrt{\eps}v \|_{C^{-\eta}} <\dl \} }
\end{align*}

\noi
in $L^p(d\mu_w)$ for any $1\le p<\infty$ as $N\to \infty$. See, for example, \cite[Proposition 1.2]{OTh}.

\end{remark}

We are now ready to take the limit $N\to \infty$ in \eqref{CHA01}.

\begin{lemma}\label{LEM:LNT}
There exists $\eps_0>0$ and $\dl>0$ such that for any function $F: \mathcal{C}^{-\eta}(\T^2;\R) \to \R$ with at most polynomial growth rate, we have 
\begin{align}
&\int_{\{ \textup{dist}_{\mathcal{C}^{-\eta} }(\phi,\mathcal{M}) <\dl  \}   } F(\phi) \exp\bigg\{ -\frac 1\eps \mathcal{V}(\phi)\bigg\} \mu_{\eps}(d\phi) \notag \\
&=\sum_{w \in \M}\Bigg[\dr_{\textup{re}}(w)  \int_{ \{ \| \sqrt{\eps}v \|_{C^{-\eta}} < \dl   \}  } F(w+\sqrt{\eps}v ) \notag \\
&\hphantom{XXXXXXXXX} \times \exp\Big\{-\frac \eps4 \H_4(v,c_{w},c) -\sqrt{\eps} \H_3(v,c_{w},c)\cdot w  \Big\}\mu_{w}(dv)  \Bigg]
\label{limN}
\end{align}

\noi
for every $0<\eps\le \eps_0$, where $\dr_{\textup{re}}(w)$, $\H_3$, and $\H_4$ are defined in \eqref{REDET0}, \eqref{H3c}, and \eqref{H4c}, respectively.

\end{lemma}

\begin{proof}
It follows from Lemmas \ref{LEM:expN}, \ref{LEM:redet}, \eqref{H3c}, \eqref{H4c}, Lemma \ref{LEM:UNIF}, and Remark \ref{REM:UNICON}, we can take the limit $N\to \infty$ in \eqref{CHA01}. 
\end{proof}

\begin{remark}\rm
Thanks to Lemma \ref{LEM:LNT}, the coefficients appearing in the expansion in Theorem \ref{THM:1} or \eqref{coeff} converge to finite limits as the ultraviolet cutoff $\P_N$  is removed. Namely, we can obtain the ultraviolet stability.
\end{remark}

\subsection{Concentration of the measures on model space}\label{SUBSEC:CMM}
In this subsection, we study the concentration of the measures \eqref{H3UN} on model space (Lemma \ref{LEM:out1}). This is the key step to get the error estimates of a remainder term in the asymptotic expansion for the $\Phi^4_2$-measure (see Lemma \ref{LEM:Remain}).

We define the model space as follows: for $w\in \M$ and $v\in C^{-\eta}$
\begin{align}
\VVert{\mathbb{V}_w }:= \| v \|_{C^{-\eta}} +\| :\! v^2 \!:_w \|_{C^{-\eta}}^{\frac 12}+\| :\! v^3 \!:_w \|_{C^{-\eta}}^{\frac 13}+\| :\! v^4 \!:_w \|_{C^{-\eta}}^{\frac 14},
\label{MDWI}
\end{align}

\noi
where 
\begin{align*}
\mathbb{V}_w=\mathbb{V}_w[v]:=(v, \; :\! v^2 \!:_w \;, \; :\! v^3 \!:_w  \;, \; :\! v^4 \!:_w )
\end{align*}

\noi
is an enhanced data set. Recall that the Wick powers are taken with respect to the new reference measure $\mu_w(dv)$ with covariance $(-\Dl+\nb^{(2)}V(w))^{-1}$. We first introduce the following Fernique’s theorem on the model space.

\begin{lemma}\label{LEM:Fern}
There exists $\al >0$ such that for any $w \in \M$, we have
\begin{align*}
\int \exp\Big\{\al \VVert{\mathbb{V}_w}^2 \Big\} \mu_w(dv)<\infty
\end{align*}

\noi
where $\mu_w(dv)$ is the Gaussian measure with covariance $(-\Dl+\nb^{(2)}V(w) )^{-1}$.
In particular, we have
\begin{align*}
\int \exp\Big\{\al \| :\! v^k \!:_w \|_{C^{- \eta }}^{\frac{2}{k}} \Big\} d\mu_w(v)<\infty
\end{align*}

\noi
for every $k \ge 1$.

\end{lemma}

For the proof of Lemma \ref{LEM:Fern}, see \cite[Theorem 2]{FO}  or \cite[Appendix D]{FK}. We are now ready to prove the concentration on the model space.

\begin{lemma}\label{LEM:out1}
Let $\dl>0$, $\eta>0$ and $F: \mathcal{C}^{-\eta}(\T^2;\R) \to \R$ be a function of at most polynomial growth. Then, there exists a constant $c>0$ and  $\eps_0 >0$ such that 
for $0<\eps\le \eps_0$, we have
\begin{align}
&\int_{ \big\{ \VVert{\mathbb{V}[\sqrt{\eps}v]   } \ge \dl, \; \| \sqrt{\eps}v \|_{C^{-\eta}} < \dl   \big\}    } F(w+\sqrt{\eps}v ) \exp\Big\{-\frac \eps4 \H_4(v,c_{w},c) -\sqrt{\eps} \H_3(v,c_{w},c)\cdot w  \Big\}\mu_{w}(dv) \notag \\
&\les  \exp\Big\{ -\frac{c\dl^2}{\eps}  \Big\},
\label{H3UN}
\end{align}

\noi
for every $w\in \M$, where $\H_4$ and $\H_3$ are as  \eqref{H3c} and \eqref{H4c}.
\end{lemma}

\begin{proof}
By using H\"older's inequality with $\frac{1}{p}+\frac{1}{q}=1$, where $1< p,q<\infty$, we have that for any $w \in \M$
\begin{align}
&\bigg| \int_{ \big\{ \VVert{\mathbb{V}[\sqrt{\eps}v]   } \ge \dl   \big\}    } F(w+\sqrt{\eps}v ) \exp\Big\{-\frac \eps4 \H_4(v,c_{w},c) -\sqrt{\eps} \H_3(v,c_{w},c)\cdot w  \Big\}\mu_{w}(dv)  \bigg| \notag \\
&\les \bigg(  \int_{ \big\{ \VVert{\mathbb{V}[\sqrt{\eps}v]   } \ge \dl   \big\}  } |F(w+\sqrt{\eps}v )|^q d\mu_{w}(v) \bigg)^{\frac 1q}  \notag \\
&\hphantom{X}\times \bigg(  \int_{ \big\{ \VVert{\mathbb{V}[\sqrt{\eps}v]   } \ge \dl   \big\}    }  \exp\Big\{-\frac { p \eps}{4} \H_4(v,c_{w},c) -p\sqrt{\eps} \H_3(v,c_{w},c)\cdot w  \Big\}\mu_{w}(dv)  \bigg)^{\frac 1p}.
\label{J15}
\end{align}

\noi
By exploiting the polynomial growth rate of $F$ and the restriction $\big\{ \VVert{\mathbb{V}[\sqrt{\eps}v]   } \ge \dl   \big\}$, we get
\begin{align}
&  \int_{ \big\{ \VVert{\mathbb{V}[\sqrt{\eps}v]   } \ge \dl   \big\}  } |F(w+\sqrt{\eps}v )|^q \mu_{w}(dv) \notag \\
&\les \exp\Big\{ -\frac{\g \dl^2}{\eps} \Big\} \int \exp\Big\{ (\g+c\eps)  \VVert{\mathbb{V}[v]   }^2    \Big\}    \mu_w(dv).
\label{J16}
\end{align}

\noi
Thanks to Lemma \ref{LEM:Fern}, by choosing both $\g>0$ and $\eps>0$ sufficiently small, we obtain
\begin{align}
 \int \exp\Big\{ (\g+c\eps)  \VVert{\mathbb{V}[v]   }^2    \Big\}    \mu_w(dv)<\infty
\label{J17}
\end{align}

\noi
for any $w\in \M=\{-1,1\}$. It follows from Lemma \ref{LEM:UNIF} that if $0<\eps\le \eps_0$, then 
\begin{align}
\int_{ \big\{ \VVert{\mathbb{V}[\sqrt{\eps}v]   } \ge \dl   \big\}    } \exp\Big\{-\frac {p \eps }{4} \H_4(v,c_{w},c) - p\sqrt{\eps} \H_3(v,c_{w},c)\cdot w  \Big\}\mu_{w}(dv) \le C_{p,\dl}<\infty
\label{J18}
\end{align}

\noi
for any $w\in \M$. By combining \eqref{J15}, \eqref{J16}, \eqref{J17}, and \eqref{J18}, we obtain the result.

\end{proof}


Thanks to Lemma \ref{LEM:out1}, we reduce the integral in \eqref{limN} as follows
\begin{align}
&\int_{\{ \textup{dist}_{\mathcal{C}^{-\eta} }(\phi,\mathcal{M}) <\dl  \}   } F(\phi) \exp\bigg\{ -\frac 1\eps \mathcal{V}(\phi)\bigg\} \mu_{\eps}(d\phi) \notag \\
&=O\bigg(\exp\Big\{ -\frac{c\dl^2}{\eps}  \Big\} \bigg)+\sum_{w \in \M}\Bigg[\dr_{\textup{re}}(w)  \int_{ \big\{ \VVert{\mathbb{V}[\sqrt{\eps}v]   } < \dl   \big\} \cap \big\{ \| \sqrt{\eps}v \|_{C^{-\eta}} < \dl   \big\}   } F(w+\sqrt{\eps}v ) \notag \\
&\hphantom{XXXXXXXXXXXXXXX} \times \exp\Big\{-\frac \eps4 \H_4(v,c_{w},c) -\sqrt{\eps} \H_3(v,c_{w},c)\cdot w  \Big\}\mu_{w}(dv)  \Bigg]
\label{limN1}
\end{align}

\noi
for every $0<\eps\le \eps_0$, where $\eps_0$ is as in Lemma \ref{LEM:UNIF}.

For notational convenience, for any $F \in C^{k+1}(C^{-\eta}(\T^2);\R)$ with derivatives of at most polynomial growth rate,
we define  $G:\M \times C^{-\eta}(\T^2)\to \R$ as follows
\begin{align}
G(w, \sqrt{\eps}v)=F(w+\sqrt{\eps}v ) \exp\Big\{-\frac \eps4 \H_4(v,c_{w},c) -\sqrt{\eps} \H_3(v,c_{w},c)\cdot w  \Big\}
\label{W11}
\end{align}

\noi
Thanks to the regularity asumption on $F$, by using Taylor's formula, it is allowed to exapnd $G$ up to order $k$ in terms of $v$
\begin{align}
G(w, \sqrt{\eps}v):=\sum_{j=0}^k \frac{\sqrt{\eps}^j}{j!} Q_j(w,v)+R_{k+1}(w,\sqrt{\eps} v)
\label{W12}
\end{align}

\noi
where $R_{k+1}(w,\sqrt{\eps}v)$ is the Taylor remainder and $Q_j(w,v)$ is the $j$-th Fr\'echet derivative of $G(w,v)$ in the coordinate $v$
\begin{align}
Q_j(w,v):=D^{(j)}G(w,0)(v,\dots,v).
\label{FRE0}
\end{align}

\noi
To simplify, we do not highlight the dependence of $Q_j$ and $R_{k+1}$ on $F$ since it will be evident from the context. It follows from the polynomial growth of $F$, its Fr\'echet derivatives, $\H_3$, $\H_4$, and the compactness of $\M=\{-1,1\}$ that there exists $r>0$ and a constant $C>0$ independent of $w\in \M$ such that   
\begin{align}
|Q_j(w,v)| \le C  (1+\VVert{\mathbb{V}[v]} )^r
\label{POLY0}
\end{align}

\noi
for every $j=1,\dots,k$ and
\begin{align}
|R_{k+1}(w,\sqrt{\eps} v)| \le C \sqrt{\eps}^{k+1} \VVert{\mathbb{V}[v]}^r_{C^{-\eta}}\exp\Big\{\frac \eps4 \big|\H_4(v,c_{w},c)
\big| +\sqrt{\eps} \big|\H_3(v,c_{w},c)\big|  \Big\}.
\label{rem01}
\end{align}

We are now ready to prove the following expansion formula on $\big\{ \VVert{\mathbb{V}[\sqrt{\eps}v]   } < \dl   \big\} $ and its error estimate.

\begin{lemma}\label{LEM:Remain}
Let $\eta>0$ and $F \in C^{k+1}(C^{-\eta}(\T^2);\R)$ with derivatives of at most polynomial growth. Then, there exists $\dl >0$ and $\eps_0=\eps_0(\dl)$ such that
\begin{align}
&\int_{ \big\{ \VVert{\mathbb{V}[\sqrt{\eps}v]   } < \dl   \big\}  } F(w+\sqrt{\eps}v ) \exp\Big\{-\frac \eps4 \H_4(v,c_{w},c) -\sqrt{\eps} \H_3(v,c_{w},c)\cdot w  \Big\}\mu_{w}(dv) \notag \\
&=\sum_{j=0}^k \frac{\sqrt{\eps}^j}{j!}  \int_{ \big\{ \VVert{\mathbb{V}[\sqrt{\eps}v]   } < \dl   \big\}  }   Q_j(w,v)  \mu_{w}(dv)+\int_{ \big\{ \VVert{\mathbb{V}[\sqrt{\eps}v]   } < \dl   \big\}  } R_{k+1}(w, \sqrt{\eps} v) \mu_{w}(dv)
\label{ASM0}
\end{align}

\noi
where $Q_j(w,v)$ and $R_{k+1}(w, \sqrt{\eps}v)$ are given in \eqref{W12} and
\begin{align}
\bigg|  \int_{ \big\{ \VVert{\mathbb{V}[\sqrt{\eps}v]   } < \dl   \big\}  } R_{k+1}(w, \sqrt{\eps} v)  \mu_{w}(dv)  \bigg| \les  \sqrt{\eps}^{k+1}
\label{ERR1}
\end{align}

\noi
for every $0<\eps \le \eps_0$.

\end{lemma}

\begin{proof}
By plugging \eqref{W11} and \eqref{W12} into \eqref{limN1}, we obtain \eqref{ASM0}. Hence, it suffices to consider \eqref{ERR1}. Recall \eqref{rem01} and
\begin{align}
\H_4(v,c_{w},c):&=\int_{\T^2}  :\! v^4 \!:_w  dx-6(c-c_{w}) \int_{\T^2} :\! v^2 \!:_w dx \label{W21}\\
\H_3(v,c_{w},c):&=\int_{\T^2}  :\! v^3 \!:_w  dx- 3(c-c_{w}) \int_{\T^2} v dx 
\label{W22}
\end{align}

\noi
where $c-c_w$ can be understood as the limit of \eqref{diffcov}. By using the homogenity of Wick powers and the structure of the set $\big\{ \VVert{\mathbb{V}[\sqrt{\eps}v]   } < \dl   \big\} $, we have 
\begin{align}
\bigg| \frac \eps4 \int_{\T^2}  :\! v^4 \!:_w  dx   \bigg| &\les \eps \bigg| \int_{\T^2}  :\! v^4 \!:_w  dx  \bigg|^\frac{2}{4} \bigg| \int_{\T^2}  :\! v^4 \!:_w  dx  \bigg|^\frac{2}{4} \notag \\
&\les \bigg| \int_{\T^2}  :\! (\sqrt{\eps}v)^4 \!:_w  dx  \bigg|^\frac{2}{4} \| :\! v^4 \!:_w \|_{C^{-\eta}}^{\frac 24}  \notag \\
&\les \dl^2 \| :\! v^4 \!:_w \|_{C^{-\eta}}^{\frac 24}, 
\label{W3}\\
\bigg| \sqrt{\eps} \int_{\T^2}  :\! v^3 \!:_w  dx   \bigg| &\les \sqrt{\eps} \bigg|  \int_{\T^2}  :\! v^3 \!:_w  dx   \bigg|^{\frac 13}  \bigg|  \int_{\T^2}  :\! v^3 \!:_w  dx   \bigg|^{\frac 23}   \notag \\
&\les  \bigg|  \int_{\T^2}  :\!  (\sqrt{\eps}v)^3 \!:_w  dx   \bigg|^{\frac 13}  \bigg|  \int_{\T^2}  :\!  v^3 \!:_w  dx   \bigg|^{\frac 23}  \notag \\
&\les \dl \| :\! v^3 \!:_w \|_{C^{-\eta}}^{\frac 23} \label{W4},
\end{align}

\noi
and
\begin{align}
\bigg|  \eps \int_{\T^2 } :\!  v^2 \!:_w  dx   \bigg| =\bigg| \int_{\T^2 } :\!  (\sqrt{\eps}v)^2 \!:_w  dx   \bigg|\le \dl^2.
\label{W5}
\end{align}

\noi
Hence, it follows from \eqref{rem01}, \eqref{W21}, \eqref{W22}, \eqref{W3}, \eqref{W4}, \eqref{W5}, and the polynomial growth of $\VVert{\mathbb{V}[v]}^r$ that
\begin{align*}
|R_{k+1}(w,\sqrt{\eps} v)| &\les  \sqrt{\eps}^{k+1} \VVert{\mathbb{V}[v]}^r\exp\Big\{ c\dl\VVert{\mathbb{V}[v] }^2    \Big\}\\
&\les  \sqrt{\eps}^{k+1} \exp\Big\{ 2c\dl\VVert{\mathbb{V}[v] }^2    \Big\}.
\end{align*}

\noi
for any $w \in \M$. Therefore, by exploiting Lemma \ref{LEM:Fern} with choosing $\dl_0>0$ sufficiently small, we have that that for every $0<\dl \le \dl_0$ 
\begin{align*}
\bigg|  \int_{ \big\{ \VVert{\mathbb{V}[\sqrt{\eps}v]   } < \dl   \big\}  } R_{k+1}(w, \sqrt{\eps} v)  \mu_{w}(dv)  \bigg| \les  \sqrt{\eps}^{k+1}.
\end{align*}

\noi
This completes the proof of Lemma \ref{LEM:Remain}.

\end{proof}

\subsection{Proof of Theorem \ref{THM:1}}
In this subsection, we present the proof of the main result (Theorem \ref{THM:1}). 

Before we present the proof of Theorem \ref{THM:1}, we first prove the following lemma.

\begin{lemma}\label{LEM:L3}
Let $\eta>0$, $\dl>0$, $F \in C^{k+1}(C^{-\eta}(\T^2);\R)$ with derivatives of at most polynomial growth rate and let $Q_j(w,v)$ be the $j$-th Fr\'echet derivative of $G(w,v)$ in $v$ as in \eqref{FRE0}, depending on $F$. Then, there exists a sufficiently small constant $\g>0$
such that 
\begin{align*}
\int_{ \big\{ \VVert{\mathbb{V}[\sqrt{\eps}v]   } \ge \dl   \big\}  } Q_j(w,v) \mu_w(dv) \les  \exp\Big\{  -\frac{\g \dl^2}{\eps} \Big\}
\end{align*}

\noi
for all $ w\in \M$.

\end{lemma}

\begin{proof}
By exploiting the condition $\big\{ \VVert{\mathbb{V}[\sqrt{\eps}v]   } \ge \dl   \big\} $ and the polynomial growth rate \eqref{POLY0}, we have
\begin{align*}
\bigg|\int_{ \big\{ \VVert{\mathbb{V}[\sqrt{\eps}v]   } \ge \dl   \big\}  } Q_j(w,v) \mu_w(dv) \bigg|& \les \exp\Big\{-\frac{\g \dl^2}{\eps} \Big\} \int Q_j(w,v) \exp\Big\{ \g \VVert{\mathbb{V}[v] }^2    \Big\} \mu_{w}(dv)\\
&\les \exp\Big\{-\frac{\g \dl^2}{\eps} \Big\} \int \exp\Big\{ 2\g \VVert{\mathbb{V}[v] }^2    \Big\} \mu_{w}(dv). 
\end{align*}

\noi
for any $\g>0$. By choosing $\g>0$ sufficiently small and using Lemma \ref{LEM:Fern}, we have
\begin{align*}
\int \exp\Big\{ 2\g \VVert{\mathbb{V}[v] }^2    \Big\} \mu_{w}(dv)<\infty.
\end{align*}

\noi
This completes the proof of Lemma \ref{LEM:L3}.
\end{proof}

We are now ready to prove Theorem \ref{THM:1}.

\begin{proof}[Proof of Theorem \ref{THM:1}]
From Proposition \ref{PROP:Var}, there exists $c=c(\dl)>0$ and sufficiently small $\eps_0 >0 $ such that 
\begin{align}
&\int F(\phi)\exp\Big\{-\frac 1\eps \mathcal{V}(\phi)  \Big\} \mu_\eps(d\phi) \notag \\
&=\int_{\{ \textup{dist}_{\mathcal{C}^{-\eta} }(\phi,\mathcal{M}) <\dl  \} } F(\phi) \exp\bigg\{ -\frac 1\eps \mathcal{V}(\phi)\bigg\} \mu_{\eps}(d\phi)\notag \\
&\hphantom{X} +  \int_{\{ \textup{dist}_{\mathcal{C}^{-\eta} }(\phi,\mathcal{M}) \ge \dl  \} } F(\phi) \exp\bigg\{ -\frac 1\eps \mathcal{V}(\phi)\bigg\} \mu_{\eps}(d\phi) \notag \\
&=O\bigg(\exp\Big\{ -\frac{c\dl^2}{\eps}  \Big\} \bigg) + \int_{\{ \textup{dist}_{\mathcal{C}^{-\eta} }(\phi,\mathcal{M}) <\dl  \}   } F(\phi) \exp\bigg\{ -\frac 1\eps \mathcal{V}(\phi)\bigg\} \mu_{\eps}(d\phi) \notag \\ 
\label{CZZ1}
\end{align}

\noi
for every $0< \eps \le \eps_0 $. Thanks to Lemmas \ref{LEM:LNT} and \ref{LEM:out1}, we can express and reduce the last integral in \eqref{CZZ1} as follows
\begin{align}
&\int_{\{ \textup{dist}_{\mathcal{C}^{-\eta} }(\phi,\mathcal{M}) <\dl  \}   } F(\phi) \exp\bigg\{ -\frac 1\eps \mathcal{V}(\phi)\bigg\} \mu_{\eps}(d\phi) \notag \\
&=O\bigg(\exp\Big\{ -\frac{c\dl^2}{\eps}  \Big\} \bigg)+\sum_{w \in \M}\Bigg[\dr_{\textup{re}}(w)  \int_{ \big\{ \VVert{\mathbb{V}[\sqrt{\eps}v]   } < \dl   \big\} } F(w+\sqrt{\eps}v ) \notag \\
&\hphantom{XXXXXXXXX} \times \exp\Big\{-\frac \eps4 \H_4(v,c_{w},c) -\sqrt{\eps} \H_3(v,c_{w},c)\cdot w  \Big\}\mu_{w}(dv)  \Bigg]
\label{CZZ2}
\end{align}

\noi
for every $0<\eps\le \eps_0$ and $0<\dl\le \dl_0$, where $\dl_0$ comes from Lemma \ref{LEM:expN}. By combining Lemmas \ref{LEM:Remain} and \ref{LEM:L3}, we have  
\begin{align}
&\sum_{w \in \M}\Bigg[\dr_{\textup{re}}(w)  \int_{ \big\{ \VVert{\mathbb{V}[\sqrt{\eps}v]   } < \dl   \big\} } F(w+\sqrt{\eps}v )
\exp\Big\{-\frac \eps4 \H_4(v,c_{w},c) -\sqrt{\eps} \H_3(v,c_{w},c)\cdot w  \Big\}\mu_{w}(dv)  \Bigg] \notag \\
&=\sum_{w\in \M} \dr_{\textup{re}}(w) \Bigg\{ \sum_{j=0}^k \frac{\sqrt{\eps}^j}{j!}  \int_{ \big\{ \VVert{\mathbb{V}[\sqrt{\eps}v]   } < \dl   \big\}  }   Q_j(w,v)  \mu_{w}(dv)+\int_{ \big\{ \VVert{\mathbb{V}[\sqrt{\eps}v]   } < \dl   \big\}  } R_{k+1}(w, \sqrt{\eps} v) \mu_{w}(dv)  \Bigg\} \notag \\
&=\sum_{w\in \M} \dr_{\textup{re}}(w) \Bigg\{ \sum_{j=0}^k \frac{\sqrt{\eps}^j}{j!}  \int Q_j(w,v)  \mu_{w}(dv)\bigg\}+O\Big(\exp\Big\{  -\frac{\g \dl^2}{\eps} \Big\}\Big) +O\big(\eps^{\frac{k+1}{2}}\big) \notag \\
&=\sum_{j=0}^k a_j \eps^{\frac{j}{2}} +O\big(\eps^{\frac{k+1}{2}}\big)
\label{CZZ3}
\end{align}

\noi
where
\begin{align}
a_j=\sum_{w\in \M}   \dr_{\textup{re}}(w)\frac{1}{j!}\int Q_j(w,v)  \mu_{w}(dv).
\label{coeff}
\end{align}

\noi
It follows from \eqref{CZZ1}, \eqref{CZZ2}, and \eqref{CZZ3} that
\begin{align*}
\int_{\{ \textup{dist}_{\mathcal{C}^{-\eta} }(\phi,\mathcal{M}) <\dl  \}   } F(\phi) \exp\bigg\{ -\frac 1\eps \mathcal{V}(\phi)\bigg\} \mu_{\eps}(d\phi)=\sum_{j=0}^k a_j \eps^{\frac{j}{2}} +O\big(\eps^{\frac{k+1}{2}}\big),
\end{align*}

\noi
from which we obtain the result.

\end{proof}

\subsection{Law of large numbers and central limit theorem}
In this subsection, we prove the law of large numbers and central limit theorem of $\Phi^4_2$-measure. We first present the proof of the law of large numbers.

\begin{proof}[Proof of Theorem \ref{THM:2}]
Let $F$ be a continuous functional on $C^{-\eta}(\T^2)$, with at most polynomial growth. Thanks to Proposition \ref{PROP:Var}, we have 
\begin{align}
&\int F(\phi)\exp\Big\{ -\frac{1}{\eps} \mathcal{V}(\phi)   \Big\}  \mu_\eps(d\phi) \notag\\
&=O\bigg(\exp\Big\{ -\frac{c\dl^2}{\eps}  \Big\} \bigg) + \int_{\{ \textup{dist}_{\mathcal{C}^{-\eta} }(\sqrt{\eps} \phi,\mathcal{M}) <\dl  \}   } F(\phi) \exp\bigg\{ -\frac 1\eps \mathcal{V}( \sqrt{\eps}\phi)\bigg\} \mu(d\phi). 
\label{JS1}
\end{align}

\noi
It follows from Lemmas \ref{LEM:LNT} and \ref{LEM:out1} that 
\begin{align}
&\int_{\{ \textup{dist}_{\mathcal{C}^{-\eta} }(\sqrt{\eps} \phi,\mathcal{M}) <\dl  \}   } F(\phi) \exp\bigg\{ -\frac 1\eps \mathcal{V}( \sqrt{\eps}\phi)\bigg\} \mu(d\phi)  \notag \\
&=\sum_{w \in \M}\Bigg[\dr_{\textup{re}}(w)  \int_{ \{ \| \sqrt{\eps}v \|_{C^{-\eta}} < \dl   \}  } F(w+\sqrt{\eps}v )\exp\Big\{-\frac \eps4 \H_4(v,c_{w},c) -\sqrt{\eps} \H_3(v,c_{w},c)\cdot w  \Big\}\mu_{w}(dv)  \Bigg] \notag \\
&=O\big( e^{-\frac{c\dl^2}{\eps}} \big) +\sum_{w \in \M}\Bigg[\dr_{\textup{re}}(w)  \int_{ \big\{ \VVert{\mathbb{V}[\sqrt{\eps}v]   } < \dl   \big\} } F(w+\sqrt{\eps} v) \notag \\
&\hphantom{XXXXXXXXXXXXXX}\times \exp\Big\{-\frac \eps4 \H_4(v,c_{w},c) -\sqrt{\eps} \H_3(v,c_{w},c)\cdot w  \Big\}\mu_{w}(dv)  \Bigg].
\label{JS2}
\end{align}

\noi
On the set $ \big\{ \VVert{\mathbb{V}[\sqrt{\eps}v]   } < \dl   \big\}$, by proceeding with \eqref{W3}, \eqref{W4}, \eqref{W5}, and the polynomial growth of $F(w+\sqrt{\eps}v)$ in $v$, we have 
\begin{align*}
F(w+\sqrt{\eps}v)\exp\Big\{\frac \eps4 \big| \H_4(v,c_{w},c)\big| +\sqrt{\eps} \big|\H_3(v,c_{w},c)\cdot w \big|  \Big\} &\les  F(w+\sqrt{\eps}v) \exp\Big\{ c\dl\VVert{\mathbb{V}[v] }^2    \Big\} \\
& \les \exp\Big\{ 2c\dl\VVert{\mathbb{V}[v] }^2    \Big\}
\end{align*}

\noi
for some constant $c>0$. By Lemma \ref{LEM:Fern} and choosing $\dl>0$ sufficiently small, we have
\begin{align*}
\int \exp\Big\{ 2c\dl\VVert{\mathbb{V}[v] }^2    \Big\} \mu_w(dv)<\infty.
\end{align*}

\noi
Also, for any fixed $v \in C^{-\eta}$, we have $F(w+\sqrt{\eps}v)\to F(w)$, $\frac \eps4 \H_4(v,c_{w},c)+ \sqrt{\eps} \H_3(v,c_{w},c) \to 0$,  and $\ind_{ \big\{ \VVert{\mathbb{V}[\sqrt{\eps}v]   } < \dl   \big\} }(v) \to 1$ as $\eps \to 0$ since $F$ is a continuous functional on $C^{-\eta}(\T^2)$. Hence, by using the dominated convergence theorem, we have
\begin{align}
&\lim_{\eps \to 0}\sum_{w \in \M}\Bigg[\dr_{\textup{re}}(w)  \int_{ \big\{ \VVert{\mathbb{V}[\sqrt{\eps}v]   } < \dl   \big\} } F(w+\sqrt{\eps}v )\exp\Big\{-\frac \eps4 \H_4(v,c_{w},c) -\sqrt{\eps} \H_3(v,c_{w},c)\cdot w  \Big\}\mu_{w}(dv)  \Bigg] \notag \\
&=\sum_{w \in \M} \int F(w) \dr_{\textup{re}}(w) \mu_w(dv) \notag \\
&=\sum_{w \in \M}  F(w) \dr_{\textup{re}}(w). 
\label{JS3}
\end{align}

\noi
By following \eqref{JS1}, \eqref{JS2}, \eqref{JS3} with $F=1$, we have
\begin{align}
\lim_{\eps \to 0}Z_\eps&=\lim_{\eps \to 0} \int \exp\Big\{ -\frac{1}{\eps} \mathcal{V}(\phi)   \Big\}  \mu_\eps(d\phi)=\sum_{w \in \M}  \dr_{\textup{re}}(w).
\label{JS31}
\end{align}

\noi
It follows from \eqref{JS1}, \eqref{JS2}, \eqref{JS3}, and \eqref{JS31} that 
\begin{align}
&\lim_{\eps \to 0}\int F(\phi) \exp\Big\{ -\frac{1}{\eps} \mathcal{V}(\phi)   \Big\}  \rho_\eps(d\phi) \notag \\
&=\lim_{\eps \to 0}\int F(\phi) \exp\Big\{ -\frac{1}{\eps} \mathcal{V}(\phi)   \Big\}  \frac{\mu_\eps(d\phi)}{Z_\eps}\notag  \\
&=\sum_{w\in \M} b(w) F(w) 
\end{align}

\noi
where 
\begin{align}
b(w):=\frac{ \dr_{\textup{re}}(w)}{\sum_{\bar w\in \M}  \dr_{\textup{re}}(\bar w)},
\label{bw}
\end{align}

\noi
from which we obtain the result.

\end{proof}

\begin{remark}\rm
From \eqref{W12} and \eqref{coeff}, the leading order term in \eqref{asymexp} is $a_0=\sum_{w \in \M} \dr_{\text{re}}(w) F(w)$.
Regarding the partition function $Z_\eps$ \eqref{PART12}  of $\Phi^4_2$-measure   
(i.e.~with $F=1$ in \eqref{W11}), the leading order term is $a_0=\sum_{w \in \M} \dr_{\text{re}}(w)$. Therefore, if we assume that $F \in C^{1}(C^{-\eta}(\T^2);\R)$ with
$F$ and its derivatives of polynomial growth, then it follows from Theorem \ref{THM:1} that 
\begin{align}
\int F(\phi)\exp\Big\{-\frac 1\eps \mathcal{V}(\phi)  \Big\} \mu_\eps(d\phi)&= \sum_{w \in \M} \dr_{\text{re}}(w) F(w) +O(\eps^{\frac 12}), \label{YY1}\\
Z_\eps=\int \exp\Big\{-\frac 1\eps \mathcal{V}(\phi)  \Big\} \mu_\eps(d\phi)&= \sum_{w \in \M} \dr_{\text{re}}(w) +O(\eps^{\frac 12}) \label{YY2}.
\end{align}

\noi
Hence, from \eqref{YY1} and \eqref{YY2}, we have
\begin{align*}
\lim_{\eps \to 0}\int F(\phi)\exp\Big\{-\frac 1\eps \mathcal{V}(\phi)  \Big\} \rho_\eps(d\phi)&=\lim_{\eps \to 0}\int F(\phi)\exp\Big\{-\frac 1\eps \mathcal{V}(\phi)  \Big\} \frac{\mu_\eps(d\phi)}{Z_\eps}\\
&=\sum_{w\in \M} b(w)F(w)
\end{align*}

\noi
where $b(w)$ is as in \eqref{bw}. We, however, point out that compared to the above proof ($F$ is a continuous functional on $C^{-\eta}(\T^2)$) , we have to assume the additional regularity of $F$ (i.e.~$F \in C^{1}(C^{-\eta}(\T^2);\R)$).
\end{remark}

We next prove the central limit theorem.

\begin{proof}[Proof of Theorem \ref{THM:3}]
From Proposition \ref{PROP:Var}, we have 
\begin{align}
&\int F\big(\sqrt{\eps}^{-1}(\phi-\pi(\phi) )   \big)\exp\Big\{ -\frac{1}{\eps} \mathcal{V}(\phi)   \Big\}  \mu_\eps(d\phi) \notag \\
&=\int F\big(\sqrt{\eps}^{-1}(\sqrt{\eps}\phi-\pi(\sqrt{\eps}\phi) )   \big)\exp\Big\{ -\frac{1}{\eps} \mathcal{V}(\sqrt{\eps}\phi)   \Big\}  \mu(d\phi) \notag \\
&=O\big( e^{-\frac{c\dl^2}{\eps}} \big) + \int_{\{ \textup{dist}_{\mathcal{C}^{-\eta} }(\sqrt{\eps} \phi,\mathcal{M}) <\dl  \}   } F\big(\sqrt{\eps}^{-1}(\sqrt{\eps}\phi-\pi(\sqrt{\eps}\phi) )   \big) \exp\bigg\{ -\frac 1\eps \mathcal{V}( \sqrt{\eps}\phi)\bigg\} \mu(d\phi). 
\label{SB1}
\end{align}

\noi
For $\dl>0$ small enough, we have $\pi(\sqrt{\eps}v+w)=w$ for all $\|\sqrt{\eps}v \|<\dl$.
By proceeding as in Lemma \ref{LEM:expN} and \ref{LEM:LNT} with the change of variable $\phi\mapsto \phi+\sqrt{\eps}^{-1}w$, choosing $\dl>0$ sufficiently small, and using Lemma \ref{LEM:out1}, we have 
\begin{align}
&\int_{\{ \textup{dist}_{\mathcal{C}^{-\eta} }(\sqrt{\eps} \phi,\mathcal{M}) <\dl  \}   } F\big(\sqrt{\eps}^{-1}(\sqrt{\eps}\phi-\pi(\sqrt{\eps}\phi) )   \big) \exp\bigg\{ -\frac 1\eps \mathcal{V}( \sqrt{\eps}\phi)\bigg\} \mu(d\phi)  \notag \\
&=\sum_{w\in \M}
\int_{ \{ \| \sqrt{\eps}\phi-w \|_{C^{-\eta}}<\dl \} } F\big(\sqrt{\eps}^{-1}(\sqrt{\eps}\phi-\pi(\sqrt{\eps}\phi) )   \big) \exp\bigg\{ -\frac 1\eps \mathcal{V}( \sqrt{\eps}\phi)\bigg\} \mu(d\phi) \notag \\
&=\sum_{w \in \M}\Bigg[\dr_{\textup{re}}(w)  \int_{ \{ \| \sqrt{\eps}v \|_{C^{-\eta}} < \dl   \}  } F(v )\exp\Big\{-\frac \eps4 \H_4(v,c_{w},c) -\sqrt{\eps} \H_3(v,c_{w},c)\cdot w  \Big\}\mu_{w}(dv)  \Bigg] \notag \\
&=O\big( e^{-\frac{c\dl^2}{\eps}} \big) +\sum_{w \in \M}\Bigg[\dr_{\textup{re}}(w)  \int_{ \big\{ \VVert{\mathbb{V}[\sqrt{\eps}v]   } < \dl   \big\} } F(v ) \notag \\
&\hphantom{XXXXXXXXXXXXXX}\times \exp\Big\{-\frac \eps4 \H_4(v,c_{w},c) -\sqrt{\eps} \H_3(v,c_{w},c)\cdot w  \Big\}\mu_{w}(dv)  \Bigg].
\label{SB2}
\end{align}

\noi
On the set $ \big\{ \VVert{\mathbb{V}[\sqrt{\eps}v]   } < \dl   \big\}$, by following \eqref{W3}, \eqref{W4}, \eqref{W5}, and the polynomial growth of $F(v)$, we have 
\begin{align*}
F(v)\exp\Big\{\frac \eps4 \big| \H_4(v,c_{w},c)\big| +\sqrt{\eps} \big|\H_3(v,c_{w},c)\cdot w \big|  \Big\} &\les  F(v) \exp\Big\{ c\dl\VVert{\mathbb{V}[v] }^2    \Big\} \\
& \les \exp\Big\{ 2c\dl\VVert{\mathbb{V}[v] }^2    \Big\}
\end{align*}

\noi
for some constant $c>0$. Thanks to Lemma \ref{LEM:Fern}, by choosing $\dl>0$ sufficiently small, we have
\begin{align*}
\int \exp\Big\{ 2c\dl\VVert{\mathbb{V}[v] }^2    \Big\} \mu_w(dv)<\infty.
\end{align*}

\noi
Also, for any fixed $v \in C^{-\eta}$, we have $\frac \eps4 \H_4(v,c_{w},c)+ \sqrt{\eps} \H_3(v,c_{w},c) \to 0$ as $\eps \to 0$ and $\ind_{ \big\{ \VVert{\mathbb{V}[\sqrt{\eps}v]   } < \dl   \big\} }(v) \to 1$ as $\eps \to 0$. Hence, by using the dominated convergence theorem, we have
\begin{align}
&\lim_{\eps \to 0}\sum_{w \in \M}\Bigg[\dr_{\textup{re}}(w)  \int_{ \big\{ \VVert{\mathbb{V}[\sqrt{\eps}v]   } < \dl   \big\} } F(v )\exp\Big\{-\frac \eps4 \H_4(v,c_{w},c) -\sqrt{\eps} \H_3(v,c_{w},c)\cdot w  \Big\}\mu_{w}(dv)  \Bigg] \notag \\
&=\sum_{w \in \M} \int F(v) \dr_{\textup{re}}(w) \mu_w(dv).
\label{SB3}
\end{align}

\noi
By combining \eqref{SB1}, \eqref{SB2}, \eqref{SB3}, and \eqref{YY2}, we obtain
\begin{align}
&\lim_{\eps \to 0}\int F\big(\sqrt{\eps}^{-1}(\phi-\pi(\phi) )   \big)\exp\Big\{ -\frac{1}{\eps} \mathcal{V}(\phi)   \Big\}  \rho_\eps(d\phi) \notag \\
&=\lim_{\eps \to 0}\int F\big(\sqrt{\eps}^{-1}(\phi-\pi(\phi) )   \big)\exp\Big\{ -\frac{1}{\eps} \mathcal{V}(\phi)   \Big\}  \frac{\mu_\eps(d\phi)}{Z_\eps}\notag  \\
&=\sum_{w\in \M} \int F(v) b(w) \mu_w(dv) \notag \\
&=\int F(v) \nu(dv)
\label{FLC}
\end{align}

\noi
where $b(w)$ is as in \eqref{bw} and
\begin{align*}
\nu:=\sum_{w\in \M} b(w) \mu_w.
\end{align*}

\noi
This shows that the fluctuation in \eqref{FLC} is the $\sum_{w}b(w)$-mixture of Gaussian fluctuations $\mu_w$ in each $w\in \M$.

\end{proof}

\begin{ackno}\rm
The authors are very grateful to Ajay Chandra for pointing out several references within quantum field theory. B.G. acknowledges support by the Max Planck Society through the Research
Group ``Stochastic Analysis in the Sciences (SAiS)”. This work was co-funded by the European Union (ERC, FluCo, grant agreement No. 101088488). Views and opinions expressed are however those of the author(s) only and do not necessarily reflect those of the European Union or of the European Research Council. Neither the European Union nor the granting authority can be held responsible for them.
\end{ackno}

\end{document}